\documentclass[a4paper]{amsart}
\usepackage{amsmath,amsthm,amssymb,enumerate}
\usepackage{epsfig}
\usepackage{amssymb}
\usepackage{amsmath}
\usepackage{amssymb}
\usepackage{amsmath,amsthm}
\usepackage[latin1]{inputenc}
\usepackage{amsfonts}
\usepackage{amsxtra}
\usepackage{euscript,mathrsfs}
\usepackage{color}
\usepackage[left=3.7cm,right=3.7cm,top=4cm,bottom=3cm]{geometry}
\usepackage[citecolor=blue,colorlinks=true]{hyperref}
\allowdisplaybreaks

\usepackage{enumitem}
\setenumerate{label={\rm (\alph{*})}}

\usepackage{amsfonts}
\usepackage{amsxtra}

  \newtheorem{theorem}{Theorem}[section]
  \newtheorem{lemma}[theorem]{Lemma}
  \newtheorem{proposition}[theorem]{Proposition}
  \newtheorem{corollary}[theorem]{Corollary}

  \newtheorem{definition}[theorem]{Definition}

  \newtheorem{remark}[theorem]{Remark}

\numberwithin{equation}{section}

\newcommand{\PPln}{\mathcal{P}^{{\log}}}
\newcommand{\Div}{\divergence}
\newcommand{\ep}{\bfvarepsilon}
\newcommand{\R}{\mathbb R}
\newcommand{\N}{\mathbb N}

\newcommand{\E}{\mathbb E}
\newcommand{\p}{\mathbb P}
\newcommand{\Q}{Q}
\newcommand{\F}{\mathcal F}

\newcommand{\q}{\overline{q}}
\newcommand{\dd}{\mathrm{d}}
\newcommand{\dx}{\,\mathrm{d}x}
\newcommand{\dt}{\,\mathrm{d}t}
\newcommand{\dxt}{\,\mathrm{d}x\,\mathrm{d}t}
\newcommand{\ds}{\,\mathrm{d}\sigma}
\newcommand{\dxs}{\,\mathrm{d}x\,\mathrm{d}\sigma}

\DeclareMathOperator{\divergence}{div}

\newcommand{\bfz}{{\bold{z}}}
\newcommand{\bfE}{{\bold{E}}}
\newcommand{\bfv}{{\bold{v}}}
\newcommand{\bfS}{{\bold{S}}}
\newcommand{\bff}{{\bold{f}}}
\newcommand{\bfP}{{\bold{P}}}
\newcommand{\bfvarepsilon}{{\boldsymbol{\varepsilon}}}
\newcommand{\bfphi}{{\boldsymbol{\varphi}}}
\newcommand{\bfvarphi}{{\boldsymbol{\varphi}}}
\newcommand{\bfxi}{{\boldsymbol{\xi}}}
\newcommand{\bfu}{{\bold{u}}}
\newcommand{\bfg}{{\bold{g}}}
\newcommand{\bfw}{{\bold{w}}}
\newcommand{\bfa}{{\bold{a}}}
\newcommand{\bfb}{{\bold{b}}}
\newcommand{\bfF}{{\bold{F}}}
\newcommand{\bfC}{{\bold{C}}}
\newcommand{\bfW}{{\bold{W}}}
\newcommand{\bfX}{{\bold{X}}}
\newcommand{\bfeta}{{\boldsymbol{\eta}}}
\newcommand{\bfomega}{{\boldsymbol{\omega}}}
\newcommand{\bfbeta}{{\boldsymbol{\beta}}}
\newcommand{\bfe}{{\bold{e}}}
\newcommand{\bfSigma}{{\boldsymbol{\Sigma}}}
\newcommand{\bfmu}{{\boldsymbol{\mu}}}

\newcommand{\dif}{\mathrm{d}}

\newcommand{\mf}{\mathscr{F}}

\newcommand{\prst}{\mathbb{P}}

\newcommand{\mn}{\mathbb{N}}

\newcommand{\mt}{\mathbb{T}^n}

\renewcommand{\dif}{\operatorname{d}\!}
\newcommand{\lebe}{\operatorname{L}}
\newcommand{\sobo}{\operatorname{W}}

\newcommand{\tn}{\mathbb{T}^{n}}
\newcommand{\e}{\operatorname{e}}
\newcommand{\imag}{\operatorname{i}}

\newcommand{\hold}{C}
\newcommand{\vv}{\mathbf{v}}
\newcommand{\ff}{\mathbf{f}}

\newcommand{\uu}{\mathbf{u}}

\allowdisplaybreaks

\newcommand{\db}[1]{\textcolor{black}{  #1}}

\begin{document}


\title[Electro-rheological fluids: martingale and strong solutions]{Electro-rheological fluids under random influences: martingale and strong solutions}

\author[D. Breit]{Dominic Breit}
\address[\textsc{D. Breit}]{Department of Mathematics, Heriot-Watt University, Riccarton Edinburgh EH14 4AS, UK}
\email{d.breit@hw.ac.uk}

\author[F. Gmeineder]{Franz Gmeineder}
\address[\textsc{F. Gmeineder}]{Universit\"{a}t Bonn, Mathematisches Institut, Endenicher Allee 60, Bonn, Germany}
\email{fgmeined@math.uni-bonn.de}

\begin{abstract}
We study generalised Navier--Stokes equations governing the motion of an electro-rheological fluid subject to
stochastic perturbation. Stochastic effects are implemented through (i) random initial data, (ii) a forcing term in the momentum equation represented by a
multiplicative white noise and (iii) a random character of the variable exponent $p=p(\omega,t,x)$ (as a result of a random electric field). We show the existence of a weak martingale solution provided the variable exponent satisfies $p\geq p^->\frac{3n}{n+2}$ ($p^->1$ in two dimensions). Under additional assumptions we obtain also stochastically strong solutions. 
\end{abstract}

\subjclass[2010]{60H15, 35R60, 76D03,  35Q30}
\keywords{Electro-rheological fluids, stochastic Navier-Stokes equations, martingale solution, pathwise solution}

\date{\today}

\maketitle
\setcounter{tocdepth}{1}


\section{Introduction}

Electro-rheological fluids are special smart fluids
which change their material properties due to the application of an
electric field firstly observed by Winslow \cite{Wi} in 1949. Since then a vast
development in the chemical constitution of electro-rheological fluids has taken place
and nowadays dramatic changes by a
factor of $10^3$ in 1ms in the viscosity are possible.
This provides the opportunity for the gainful exploitation of this fact in technological applications for instance in clutches, shock absorbers, valves, actuators and exercise equipment.

The simplest approach for the
modelling of such suspensions is to treat them in a homogenised sense within the framework of continuum
mechanics and in this respect, we restrict ourself to incompressible fluids with density $\varrho>0$. The conservation
of mass and the balance of linear momentum are given by
\begin{align}\label{0.1}
\left\{\begin{array}{rc}
\partial_t (\varrho\bfv)-\Div \bfS=-\Div(\varrho\bfv\otimes\bfv)-\nabla \pi+\varrho\bff+\bff_e& \mbox{in $Q$,}\\
\Div \bfv=0\qquad\qquad\qquad\,\,\,\,& \mbox{in $Q$},
\end{array}\right.
\end{align}
where $Q=(0,T)\times\mathcal O$ denotes the parabolic cylinder ($\mathcal O$ is a bounded domain in $\R^n$, $n=2,3$) and $\otimes$ is the tensor product in $\R^n$ (that is we have $\bfa\otimes\bfb=\bfa \bfb^T$ for $\bfa,\bfb\in\R^n$). Here $\bfv:Q\rightarrow\mathbb{R}^n$ is the velocity field, $\pi:Q\rightarrow\mathbb{R}$ the pressure, $\bfS:Q\rightarrow \mathbb{R}^{n\times n}$ the viscous stress tensor whereas $\bff:Q\rightarrow \R^n$ is the external mechanical body force and $\bff_e:Q\rightarrow \R^n$ the electromagnetic force. The material properties of an electro-rheological  fluid - according to Rajagopal and R{\r u}{\v z}i{\v c}ka \cite{RaRu1,RaRu2} - are described by the relation 
\begin{align}
\bfS&=\alpha_{21}\big((1+|\ep(\bfv)|^2)^{\frac{p-1}{2}}-1\big)\bfE\otimes\bfE+(\alpha_{31}+\alpha_{33}|\bfE|^2)(1+|\ep(\bfv)|^2)^{\frac{p-2}{2}}\ep(\bfv)\nonumber\\
&+\alpha_{51}(1+|\ep(\bfv)|^2)^{\frac{p-2}{2}}(\bfE\otimes\ep(\bfv)\bfE+\bfE\otimes\ep(\bfv)\bfE).
\label{0.2}
\end{align}
Here $\bfE:Q\rightarrow\R^n$ is the electric field (which is
a solution to the quasi-static Maxwell equations and is not
influenced by the motion of the fluid), $\ep(\bfv)=\tfrac{1}{2}\big(\nabla\bfv+\nabla\bfv^T\big)$ the symmetric gradient of the velocity field and $\alpha_{ij}$ are material constants. The exponent $p=p(|\bfE|^2)$ depends on the strength of the electric field (and hence on time and space) and satisfies in $Q$
\begin{align}\label{0.2b}
1<p^{-}\leq  p\leq p^+<\infty.
\end{align}
In the mathematical literature about electro-rheological fluids (starting with \cite{Ru} and \cite{Di}) it is common to study the constitutive law
\begin{align} 
  \label{0.3a}
  \bfS=\mu(1+|\ep(\bfv)|)^{p(\cdot)-2}\ep(\bfv),\quad\mu>0,
\end{align}
which contains the same mathematical difficulties as \eqref{0.2} but simplifies the calculations.
Essentially, there are two parts in the model where randomness can occur:
\begin{itemize}
\item The electromagnetic force is mainly influenced by the gradient of the electric field $\bfE$ and the electric polarization $\bfP$ such that $\bff_e=[\nabla\bfE]\bfP$.
All missing quantities which are neglected here (for instance magnetic field and magnetic polarization) can be summarized in some random perturbation. In addition, it can incorporate physical uncertainties and turbulence in the fluid motion.
\item The exponent $p$ depends of the strength of the electric field which is a solution to Maxwell's equation, the latter having been widely studied in literature.
Randomness naturally appears in the Maxwell equation (see, for instance \cite{CHL1,CHL2} for stochastic Maxwell equations), and the randomness in the Maxwell equation transfers to randomness in the exponent in the model \eqref{0.3}.  In conclusion, the assumption of a random exponent is very reasonable and required by applications.
\end{itemize}
Second, it is not possible to give an explicit formula for the exponent $p$. Its dependence on the electric field has to be determined via physical experiments. Hence some (random) derivation from the ``real'' exponent is expected.

In this respect, the aim of this paper is to give a rigorous analysis of the following stochastic model for electro-rheological fluids (without loss of generality we assume that $\varrho=1$ and $\bff_e=0$)
\begin{align}\label{0.4}
\left\{\begin{array}{rc}
\dd\bfv=\Div\bfS\dt-\Div(\bfv\otimes\bfv)\dt-\nabla\pi\dt+\bff\dt+\Phi(\bfv)\dd W
& \mbox{in $Q$,}\\
\Div \bfv=0\qquad\qquad\qquad\qquad\qquad\,\,\,\,& \mbox{in $Q$,}\\
\bfv(0)=\bfv_0\,\qquad\qquad\qquad\qquad\qquad&\mbox{ \,in $\mathcal O$,}\end{array}\right.
\end{align}
with $\bfS$ given by 
\begin{align} 
  \label{0.3}
  \bfS=\mu(1+|\ep(\bfv)|)^{p(\cdot)-2}\ep(\bfv),\quad\mu>0.
\end{align}
We suppose that the electric field $\bfE$ is given and that $p=p(\omega,t,x)$ satisfies \eqref{0.2b}. The quantity $W$ denotes a cylindrical Wiener process with values in some Hilbert space and $\Phi$ is nonlinear in $\bfv$ with linear growth, cp.~Section \ref{Probability Setup} for further details.

In the general three-dimensional case, regularity and uniqueness of solutions to \eqref{0.4}--\eqref{0.3}
is a longstanding open problem (already in the deterministic situation) even if $p\equiv2$, leading to the classical Navier--Stokes equations for Newtonian fluids. Consequently, the solution is understood weakly in space-time (in the sense of distributions) and also weakly in the probabilistic sense (i.e., the underlying probability space is part of the solution). This concept of stochastically weak solutions already appears on the level of stochastic ODEs if uniqueness fails.

As far as stochastic PDEs are concerned, a milestone was the existence
of martingale solutions to the stochastic Navier--Stokes equation (\eqref{0.4}--\eqref{0.3} with $p\equiv2$) by Flandoli-Gatarek \cite{FlGa}. Today there exists an abundant amount of literature concerning the dynamics of incompressible Newtonian fluids driven by stochastic forcing. We refer to the lecture notes by Flandoli \cite{Fl}, the monograph of Kuksin and
Shyrikian \cite{KukShi}, the survey by Romito \cite{Ro} as well as the references cited therein for a recent overview. Much less is known if other fluid types are concerned. Just very recently, an analysis of non-Newtonian fluids (see \cite{Br2,TeYo,Yo}) and compressible fluids (see \cite{BrHo} and \cite{2015arXiv150400951S}) subject to stochastic forcing started.

The analysis the system \eqref{0.4}--\eqref{0.3} brings a completely new aspect into play: a random variable exponent. As a consequence, solutions are located in a random function space generated by the a priori information
\begin{align*}
\E\left[\int_Q|\ep(\bfv)(\omega,t,x)|^{p(\omega,t,x)}\dxt\right]<\infty.
\end{align*}
So, we have
\begin{align*}
\ep(\bfv)(\omega,\cdot)\in \lebe^{p(\omega,\cdot)}(Q)\quad \text{for $\p$-a.e. }\omega\in\Omega,
\end{align*}
where
\begin{align}\label{eq:Lpx}
  \lebe^{p(\cdot)}(G)=\Big\{f\in
    \lebe^1(G):\,\,\int_{G}|f(y)|^{p(y)}\,\dd y < \infty\Big\}
\end{align}
for $G\subset \mathbb R^m$ and $p:G\rightarrow [1,\infty)$ measurable. Variable exponent Lebesgue spaces (and Sobolev spaces) as in \eqref{eq:Lpx} have been studied
extensively over the last two decades motivated by the model for electro-rheological fluids from \cite{RaRu1,RaRu2}, and we refer to \cite{DiHaHaRu} for a comprehensive treatment. As far as stochastic problems are concerned, a first analysis
of problems involving random variable exponents can be found in \cite{VaWiZi} (see also \cite{BaVaWiZi} for a previous result on a stochastic $p(t,x)$-Laplacian). In this work, the existence and uniqueness of weak solutions of a stochastic $p(\omega,t,x)$-Laplacian type equation is established by use of the variational approach, and problems connected to compactness and non-uniqueness do not occur.

The boundary conditions in the real world applications are quite
complicated and of substantial influence on the fluid motion. Nevertheless, our goal is to focus on
the effect of a random variable exponent as well as stochastic perturbations imposed through stochastic volume forces. So, for a first analysis
we consider \emph{periodic} boundary conditions, where the physical domain
is identified with the flat torus
\[
\mt = \Big([0,1]\Big|_{\{0,1\}} \Big)^n.
\]
The first main result of this paper is the existence of a weak martingale solution
to \eqref{0.4}--\eqref{0.3} under periodic boundary conditions where the variable exponent $p$ is Lipschitz continuous in $x$ and satisfies
\begin{align}\label{eq:3011}
\inf_{\Omega\times Q}p>\frac{3n}{n+2},
\end{align}
see Theorem \ref{thm:main} for the precise statement. This generalises the results
from \cite{TeYo} to the case of variable exponents. 
As a consequence of the nature of martingale solutions we are not able to describe the variable exponent as a given function defined on $\Omega\times Q$. Instead, we rather describe a probability law on $\hold^0([0,T]\times \mt)$. 

Our approach is based on a finite-dimensional Galerkin approximation and a refined stochastic compactness method involving Skorokhod's representation theorem. Since the system \eqref{0.4}--\eqref{0.3} is nonlinear in the gradient of the velocity field we have to demonstrate its compactness in $W^{1,1}(\mt)$ first.
This is achieved by fractional estimates for $\nabla\bfv$
inspired by the results from \cite[Chapter 5]{MNRR}, where deterministic problems with constant $p$ are considered. Under more restrictive assumptions on the variable exponent $p$, we are able to show pathwise uniqueness of solutions. As a consequence, we obtain pathwise solutions (see Theorem \ref{thm:main3}) using the method by Gy\"ongy-Krylov.
Eventually, we are concerned with the existence of analytically strong solutions (see Definitions \ref{def:strong} and \ref{def:strongstrong}), where equation \eqref{0.4}$_1$ holds almost everywhere in space. This is based on the existence of second derivatives of the velocity field.
Because of the non-standard growth character of \eqref{0.3} this is much more involved than the situation with constant $p$. By simply differentiating equation \eqref{0.4}$_1$ we are left with an a priori unbounded integral, cp. \eqref{eq:reg}.
This issue can be overcome by combining a parabolic interpolation as in \cite{AcMiSe} with an improved moment estimate, cp. Theorem \ref{thm:2.1'}. Consequently, we obtain weak (or even strong) pathwise solutions to \eqref{0.4}--\eqref{0.3}, see Theorem \ref{thm:main2} and Corollary \ref{cor:strong}.

The paper is organised as follows. In Section \ref{sec:framework} we present the mathematical framework, the various solution concepts to \eqref{0.4}--\eqref{0.3} as well as the main results.
In Section \ref{sec:galerkin} we study the finite-dimensional approximation
to \eqref{0.4}--\eqref{0.3} and derive uniform a prior estimates.
Section \ref{sec:weak} is dedicated to the existence of martingale solutions.
Under more restrictive assumptions on the exponent $p$, we then show existence of stochastically strong solutions. In the final section we establish the existence of analytically strong solutions subject to suitable additional assumptions imposed on the data. 

\subsection*{Acknowledgments.} The authors gratefully acknowledge support through the Edinburgh Mathematical Society during a stay of the second author at Heriot Watt University Edinburgh in November 2016, where this work had been commenced.\\
The authors would like to thank the referee for the careful reading of the manuscript and the valuable suggestions.

\section{Framework and Main Results}
\label{sec:framework}
\subsection{Function Space Setup}\label{sec:functionspaces}
In this section we briefly introduce the function spaces to be dealt with in the main part of the paper. Incorporating the periodic boundary conditions, all spatial function spaces are defined on the torus $\tn$. 
Specifically, we define for $0<\kappa<\infty$ and $1<q<\infty$ the corresponding Bessel--Sobolev spaces by 
\begin{align*}
& \sobo^{\kappa,q}(\tn):=\left\{ v\colon\tn\to\R^{n}\colon\; \|v\|_{\kappa,q}^{q}:= \sum_{k\in\mathbb{Z}}\langle k\rangle^{\kappa q}|c_{k}(v)|^{q} <\infty\right\}, \\
& \sobo_{\Div}^{\kappa,q}(\tn):=\sobo^{\kappa,q}(\tn)^{n}\cap\{v\in\lebe^{1}(\tn;\R^{n})\colon\,\Div(v)=0\,\text{in the sense of distributions}\},
\end{align*}
where $\langle\xi\rangle:=\sqrt{1+|\xi|^{2}}$ and $c_{k}(v)$ are the Fourier coefficients of $v\colon\tn\to\R^{n}$ with respect to the standard Fourier basis $\{x\mapsto\exp{(\db{2\pi}\imag k\cdot x)}\}_{k}$. Given a real Banach space $(X,\|\cdot\|)$, we moreover introduce the fractional Sobolev space $\sobo^{\kappa,q}(0,T;X)$ 
as the collection of all measurable $v\colon [0,T]\to X$ such that $v\in\lebe^{q}(0,T;X)$ (in the sense of Bochner integrability) and  
\begin{align*}
[v]_{\kappa,q}:=\int_{0}^{T}\int_{0}^{T}\frac{\|v(s)-v(t)\|_{X}^{q}}{|s-t|^{1+\kappa q}}\dif s\dif t<\infty.
\end{align*}
Let us note that the former space could be defined similarly by use of Fourier coefficients, however, we refrained from doing so to emphasize the non--periodicity with respect to time. \\
\db{We continue with a brief introduction of variable exponent Lebesgue spaces. For a given continuous function $p:Q\rightarrow[1,\infty)$ with $Q=(0,T)\times\mathbb T^n$ we define
the variable exponent Lebesgue space $\lebe^{p(\cdot)}(Q)$ by
\begin{align*}
  \lebe^{p(\cdot)}(Q)  =\sup\Big\{f\in
    \lebe^1(Q):\,\,\int_{Q}|f(t,y)|^{p(t,y)}\,\dd y\dt < \infty\Big\}.
\end{align*}
It is a Banach space together with the Luxemburg norm
\begin{align}\label{normptx}
\|f\|_{p(\cdot)}=\inf\Big\{k\geq0:\,\,\int_{Q}\Big|\frac{f(t,y)}{k}\Big|^{p(t,y)}\,\dd y\dt\leq 1\Big\}.
\end{align}
For most of the interesting functional analytical properties of $\lebe^{p(\cdot)}(Q)$ some mild regularity of $p$ is needed.
We say that a function $g: Q \to \mathbb R$ is $\log$-H{\"o}lder continuous in $Q$ if there exists a constant
$c \geq 0$ such that
\begin{align*}
  |g (X)-g (Y)| &\leq \frac{c}{\log
    (e+1/|X-Y|)},
\end{align*}
for all $X\not=Y\in Q$.
The smallest such constant~$c$ is the $\log$-H{\"o}lder constant
of~$g$. We define $\PPln(Q)$ to consist of those exponents $p\in
\lebe^1(Q)$ for which $\frac{1}{p} \,:\,  Q \to (0,1]$ is $\log$-H{\"o}lder
continuous. The norm $\|p\|_{\PPln(Q)}$ is the $\log$-H{\"o}lder constant
of~$1/p$. 
For $p\in \PPln(Q)$ almost all properties of the classical Lebesgue spaces extend to $\lebe^{p(\cdot)}(Q)$. In particular smooth functions are dense with respect to the norm given in \eqref{normptx}.}

Lastly, we shall sometimes surpress the target space and write, e.g., $\sobo^{\kappa,q}(\mathbb{T}^{n})$ instead of $\sobo^{\kappa,q}(\mathbb{T}^{n})^{n}$. However, no ambiguities will arise from this.

\subsection{Probability Setup}
\label{Probability Setup}
Let $(\Omega,\mf,\p)$ be a probability space endowed with a filtration $(\mf_t)=(\mf_t)_{t\geq0}$ which is a nondecreasing
family of sub-$\sigma$-fields of $\mf$, i.e., $\mf_s\subset\mf_t$ for $0\leq s\leq t\leq T$. We further assume that $(\mf_t)_{t\geq0}$ is right-continuous and $\mf_0$ contains all the $\p$-negligible events in $\mf$.

For a Banach space $(X,\|\cdot\|_X)$ and corresponding Borel $\sigma$-algebra $\mathfrak{B}(X)$, we denote by for $1\leq p<\infty$ by $\lebe^p(\Omega;X)$ the Banach space of all measurable functions $v:(\Omega,\mf)\rightarrow (X,\mathfrak B(X))$ such that
\begin{align*}
\E\big[\|v\|_X^p\big]=\int_\Omega\|v\|_X^p\,\dd\p <\infty.
\end{align*}
Let $\mathfrak U$ be a Hilbert space with orthonormal basis $(e_k)_{k\in\N}$ and let $\lebe_2(\mathfrak U,\lebe^2(\mt))$ be the set of Hilbert-Schmidt operators from $\mathfrak U$ to $\lebe^2(\mt)$. Moreover, define the auxiliary space $\mathfrak U_0\supset \mathfrak U$  as
\begin{align}\label{eq:U0}
\begin{aligned}
\mathfrak U_0&:=\left\{e=\sum_{k=1}^{\infty} \alpha_ke_k:\,\,\sum_{k=1}^{\infty} \frac{\alpha_k^2}{k^2}<\infty\right\},\\
\|e\|^2_{\mathfrak U_0}&:=\sum_{k=1}^\infty \frac{\alpha_k^2}{k^2},\quad e=\sum_{k=1}^{\infty} \alpha_ke_k.
\end{aligned}
\end{align}
Throughout the paper we consider a cylindrical $(\mf_t)$-Wiener process
$W=(W_t)_{t\geq0}$ which has the form
\begin{align}\label{eq:W}
\sobo=\sum_{k\in\N} \beta_k e_k
\end{align}
with a sequence $(\beta_k)$ of independent real valued $(\mf_t)$-Wiener processes. The embedding $\mathfrak U\hookrightarrow \mathfrak U_0$ is Hilbert-Schmidt and trajectories of $W$ are $\p$-a.s. continuous with values in $\mathfrak U_0$ (see \cite{PrZa}).
Now, for  $\Psi\in \lebe^2(\Omega;\lebe^2(0,T;\lebe_2(\mathfrak U,\lebe^2(\mt))))$ $(\mf_t)$-progressively measurable\footnote{We understand progressive measurability for non-continuous processes in the sense of random distributions introduced in \cite[Section 2.2]{BFHbook}.} we have that
$\int_0^t \Psi\,\dd W$
defines a $\p$-almost surely continuous $\lebe^2(\mt)$-valued $(\mf_t)$-martingale (cp.~\cite{PrZa} for stochastic calculus in infinite dimensions). Moreover, we can multiply with test-functions because
 \begin{align*}
\int_{\mt}\int_0^t \Psi\,\dd W\cdot \bfphi\dx=\sum_{k=1}^\infty \int_0^t\int_{\mt} \Psi e_k\cdot\bfphi\dx\,\dd\beta_k,\quad \bfphi\in \lebe^2(\mt),
\end{align*}
is well--defined.

In the SPDEs appearing in this paper we consider a noise coefficient $\Phi(\bfv)$
(depending on the solution $\bfv$) with values in $\lebe_2(\mathfrak U,\lebe^2(\mt))$.
We suppose the following linear growth assumptions on $\Phi$: For each $\bfz\in \lebe^2(\mt)$ there is a mapping $\Phi(\bfz):\mathfrak U\rightarrow \lebe^{2}(\mt)$ defined by $\Phi(\bfz)e_k=g_k(\bfz(\cdot))$. In particular, we suppose
that $g_k\in C^1(\R^n)$
and the following conditions for some $L\geq0$
\begin{align}\label{eq:phi}
&\sum_{k\in\N}|g_k(\bfxi)|^2 \leq L(1+|\bfxi|^2),\quad\sum_{k\in\N}|\nabla g_k(\bfxi)|^2 \leq L,\quad\bfxi\in\R^n.
\end{align}
\subsection{Martingale solutions}
Now we are in position to give a precise formulation of the meaning of a martingale solutions. We start with a weak martingale solution.
This solution is weak on both senses. Derivatives have to be understood in the sense of distributions (weak in the PDE-sense) and the underlying probability space is not a priori given but is part of the problem (weak in the probabilistic sense). Accordingly, the initial condition is given as a Borel probability measure on
$L^2_{\Div}(\mt)$. The same applies for the forcing $\bff$ which will be given
as a Borel probability measure on
$L^2(Q)$
As usual the moments of data measured via
\begin{align*}
C_r(\Lambda_0,\Lambda_\bff)
&=
\int_{\lebe^2_{\Div}(\mt)}
\big\|\bfu\big\|_{\lebe^2(\mt)}^{2r}\,\dd\Lambda_0(\bfu)+
\int_{\lebe^2(\Q)}\big\|\bfg\big\|_{\lebe^2(\Q)}^{2r}\,\dd\Lambda_\bff(\bfg)
\end{align*}
transfer to the solution.
Solutions as described above are called martingale solutions due to the connection to
the so-called Stroock--Varadhan martingale problem (see, e.g., \cite[Chap. 5.4]{KS}).
\begin{definition}[Weak martingale solution]\label{def:weak}
Let \db{$\Lambda$ be a Borel probability law on $\lebe^{2}_{\Div}(\tn)\times\lebe^{2}(Q)\times C^{0}([0,T]\times \mt)$ with marginals $\Lambda_0,\Lambda_\bff, \Lambda_p$.} 
Then a quintuple
$$\big((\Omega,\mf,(\mf_t),\p),\bfv,\bff,p,W)$$
is called a \emph{weak martingale solution} to \eqref{0.4}--\eqref{0.3} with the initial datum $\Lambda_0$, right-hand-side $\Lambda_\bff$ and exponent $\Lambda_p$ provided
\begin{enumerate}
\item $(\Omega,\mf,(\mf_t),\p)$ is a stochastic basis with a complete right-continuous filtration,
\item $W$ is an $(\mf_t)$-cylindrical Wiener process,
\item $\bff\in \lebe^2(\Omega,\F,\p;\lebe^2(\Q))$ is $(\mf_t)$-progressively measurable
\item $p\in \db{\lebe^2(\Omega,\F,\p;C^{0}([0,T]\times \mt))}$ is $(\mf_t)$-progressively measurable,
\item the velocity field satisfies $\bfv\in C_w([0,T];\lebe^2(\mt))$, $\ep(\bfv)\in \lebe^{p(\cdot)}(Q)$, $\p$-a.s. and is $(\mf_t)$-progressively measurable, 
\item we have \db{$\Lambda=\p\circ (\bfv(0),\bff,p)^{-1}$},
\item for all
 $\bfvarphi\in C^\infty_{\Div}(\mt)$ and all $t\in[0,T]$ there holds $\p$-a.s.\footnote{By $:$ we denote the inner product between matrices, that is $\bold{A}:\bold{B}=\sum_{ij}A_{ij}B_{ij}$ for $\bold{A},\bold{B}\in\R^{n\times n}$.}
\begin{align*}
\int_{\mt}\bfv(t)\cdot\bfvarphi\dx &+\int_0^t\int_{\mt}\mu(1+|\ep(\bfv)|)^{p(\cdot)-2}\ep(\bfv):\ep(\bfphi)\dxs-\int_0^t\int_{\mt}\bfv\otimes\bfv:\ep(\bfphi)\dxs\\&=\int_{\mt}\bfv(0)\cdot\bfvarphi\dx
+\int_{\mt}\int_0^t\bff\cdot\bfphi\dxs+\int_{\mt}\int_0^t\Phi(\bfv)\,\dd W\cdot \bfvarphi\dx.
\end{align*}
\end{enumerate}
\end{definition}

We obtain the following result.

\begin{theorem}[Weak martingale solution]\label{thm:main}
Suppose that
\begin{equation}\label{initial}
\int_{\lebe^{2}_{\Div}(\mt)}\big\|\bfu\big\|_{\sobo^{1,2}(\mt)}^2\,\dd\Lambda_0(\bfu)<\infty,\quad
\int_{\lebe^2(\Q)}\big\|\bfg\big\|_{L^2(0,T;\sobo^{1,2}(\mt))}^2\,\dd\Lambda_\bff(\bfg)<\infty,\\
\end{equation}
as well as $C_r(\Lambda_0,\Lambda_\bff)<\infty$ for all $1\leq r<\infty$.
Moreover, assume that 
\begin{align}\label{eq:lawp}
\Lambda_p\big\{h\in \db{\PPln(Q)}:p^-\leq h\leq p^+,\,\,\|h\|_\infty+\|\nabla h\|_\infty\leq\,c_p\big\}=1,
\end{align}
where $c_p<\infty$ and
\begin{align}\label{eq:p-p+}
\frac{3n}{n+2}<p^-\leq p^+<\frac{n+2}{n}p^-
\end{align}
\db{for some deterministic constants $p^-$ and $p^+$. Additionally, suppose that
\begin{align}\label{eq:lawp2}
\int_{\hold^0([0,T]\times \mt)}\big\|h\big\|_{\PPln(Q)}\,\dd\Lambda_p(h)<\infty
\end{align}}
Finally, assume that $\Phi$ satisfies \eqref{eq:phi}.
Then there is a weak martingale solution to \eqref{0.4}--\eqref{0.3} in the sense of Definition \ref{def:weak}. We have the energy estimate
\begin{align}\label{eq:regularityestimate1a}
\begin{split}
\E&\bigg[\sup_{t\in(0,T)}\int_{\mt} |\bfv(t)|^2\dx+\int_{\Q} |\ep(\bfv)|^{p(\cdot)}\dxt\bigg]^r\leq c\big(1+C_r(\Lambda_0,\Lambda_\bff)\big).
\end{split}
\end{align}
for \db{any $r\geq1$}.
\end{theorem}
\begin{remark}\label{rem:2.3}
Let us explain the assumptions on upper and lower bound on $p$ in \eqref{eq:p-p+}.
\begin{itemize}
\item The lower bound is the same as in the case of constant from \cite{TeYo} in the two and three dimensional case (\db{we do not consider higher dimensions as they are not of physical interest}). 
\item It will become clear from the proof of Theorem \eqref{thm:main}
that the assumption \eqref{eq:p-p+} can be relaxed to
\begin{align}\label{eq:p-p+2}
\frac{3n-4}{n}<p^-\leq p^+<np^-+4
\end{align}
provided $p^-\geq2$ (where the lower bound is redundant for $n=2,3$).
We decided for the version in \eqref{eq:p-p+} as it is physically meaningful that $p^-$ is as low as possible
whereas non-Newtonian fluids with growth-exponent higher than $p=3$ are not known (the case $p=3$ refers to the the classical Smagorinsky
model \cite{Sm}).
\end{itemize}
\end{remark}
\begin{remark}
By slightly refining our estimates it is possible to weaken the assumption in \eqref{eq:lawp} from a deterministic constant $c_p$ to a random variable $c_p$ with arbitrary high moments. This seems more realistic in view of the random character of the exponent.
\end{remark}
\begin{remark}
In contrast to the deterministic case we need assumptions between $p^-$ and $p^+$ to balance our estimates. In the deterministic case this can be avoided
by localizing the problem and arguing on a small parabolic cube where $p^-$ and $p^+$ are arbitrary close (recall that $p$ is continuous). This is not possible here
because of the random character of $p$.
\end{remark}

The method we are using in the proof of Theorem \ref{thm:main} originates
from \cite[Chap. 5]{MNRR}, where the deterministic problem with constant $p$ is studied. The key idea is to analyse fractional derivatives of the velocity gradient.
This method is only very powerful in the case of periodic boundary conditions, where a test with $\Delta \bfv^N$ ($\bfv^N$ is the Galerkin approximation of the velocity field) is possible. The situation in the two-dimensional situation is much better than the 3D case as we have
\begin{align*}
\int_{\mt}\bfv^N\otimes\bfv^N:\nabla\bfv^N\dx=0.
\end{align*}
Due to this we can expect solutions which are strong in PDE sense. Before we give a proper definition we have to introduce the pressure function (as we need a formulation which holds a.e. in space without test-functions).\\
Assume that $\big((\Omega,\mf,(\mf_t),\p),\bfv,\bff,p,W)$ is a weak martingale solution to \eqref{0.4}--\eqref{0.3} in the sense of Definition \ref{def:weak}. In particular, we have $\p$-a.s.
\begin{align*}
\int_{\mt}\bfv(t)\cdot\bfvarphi\dx &+\int_0^t\int_{\mt}\mu(1+|\ep(\bfv)|)^{p(\cdot)-2}\ep(\bfv):\ep(\bfphi)\dxs-\int_0^t\int_{\mt}\bfv\otimes\bfv:\ep(\bfphi)\dxs\\&=\int_{\mt}\bfv(0)\cdot\bfvarphi\dx
+\int_{\mt}\int_0^t\bff\cdot\bfphi\dxs+\int_{\mt}\int_0^t\Phi(\bfv)\,\dd W\cdot \bfvarphi\dx
\end{align*}
for all
 $\bfvarphi\in C^\infty_{\Div}(\mt)$ and all $t\in[0,T]$. Now, for $\bfphi\in C^\infty(\mt)$ we can insert $\bfphi-\nabla\Delta^{-1}\Div\bfphi$
and obtain
\begin{align}
\nonumber
\int_{\mt}\bfv(t)\cdot\bfvarphi\dx &+\int_0^t\int_{\mt}\mu(1+|\ep(\bfv)|)^{p(\cdot)-2}\ep(\bfv):\ep(\bfphi)\dxs-\int_0^t\int_{\mt}\bfv\otimes\bfv:\ep(\bfphi)\dxs\\
\label{eq:pressure}&=\int_{\mt}\bfv(0)\cdot\bfvarphi\dx
+\int_0^t\int_{\mt}\pi_{\mathrm{det}}\,\Div\bfphi\dxs+\int_{\mt}\int_0^t\bff\cdot\bfphi\dxs
\\
\nonumber&+\int_{\mt}\int_0^t\Phi(\bfv)\,\dd W\cdot \bfvarphi\dx+\int_{\mt}\int_0^t\Phi^\pi\,\dd W\cdot \bfvarphi\dx,
\end{align}
where 
\begin{align*}
\pi_{\mathrm{det}}&=\pi_{\mathrm{det}}^1+\pi_{\mathrm{det}}^2+\pi_{\mathrm{det}}^3,\\
\pi_{\mathrm{det}}^1&=\Delta^{-1}\Div\Div\big(\mu(1+|\ep(\bfv)|)^{p(\cdot)-2}\ep(\bfv)\big),\\
\pi_{\mathrm{det}}^2&=-\Delta^{-1}\Div\Div\big(\bfv\otimes\bfv\big),\\
\pi_{\mathrm{det}}^3&=\Delta^{-1}\Div\bff,\\
\Phi^\pi&=-\nabla\Delta^{-1}\Div\Phi(\bfv).
\end{align*}
This corresponds to the stochastic pressure decomposition introduced in \cite[Chap. 3]{Br2}. However, the situation with periodic boundary conditions we are considering here is much easier as the harmonic component of the pressure disappears. From a strong solution (in the PDE-sense) we expect that
\eqref{eq:pressure} holds without the use of the test-functions, i.e. we have
\begin{align*}
\bfv(t)&=\bfv(0)+\int_0^t\Big[\Div\Big(\mu(1+|\ep(\bfv)|)^{p(\cdot)-2}\ep(\bfv)\Big)-\Div\big(\bfv\otimes\bfv\big)
-\nabla\pi_{\mathrm{det}}+\bff\Big]\ds\\&+
\int_0^t\big[\Phi(\bfv)+\Phi^\pi\big]\,\dd W
\end{align*}
$\p$-a.s. for all $t\in[0,T]$. We remark that already under the assumptions of Theorem \ref{thm:main} we have enough spatial regularity to define $\Div\big(\bfv\otimes\bfv\big)$ as an $\lebe^1$-function (in fact $p^-\geq\frac{2n+2}{n+2}$ is required). So, the critical point is whether second derivatives of $\bfv$ exists and  $\Div\big((\kappa+|\ep(\bfv)|)^{p-2}\ep(\bfv)\big)$ is an $\lebe^1$-function. The required regularity of the pressure terms follows immediately from this and continuity properties of $\Delta^{-1}$ on Lebesgue and Sobolev spaces. Let us finally mention that regularity of $\bfv$ is usually measured via the nonlinear function $\bfF_p(\cdot,\ep(\bfv))$, where
$$\bfF_p(\omega,t,x,\bfeta)=(1+|\bfeta|)^{\frac{p(\omega,t,x)-2}{2}}\bfeta,\quad \bfeta\in \R^{n\times n}.$$
Now we are ready to define a strong martingale solution.
\begin{definition}[Strong martingale solution]\label{def:strong}
\db{Let $\Lambda$ be a Borel probability law on $\lebe^{2}_{\Div}(\tn)\times\lebe^{2}(Q)\times C^{0}([0,T]\times \mt)$ with marginals $\Lambda_0,\Lambda_\bff, \Lambda_p$.} 
Then a quintuple
$$\big((\Omega,\mf,(\mf_t),\p),\bfv,\bff,p,W)$$
is called a \emph{strong martingale solution} to \eqref{0.4}--\eqref{0.3} with the initial datum $\Lambda_0$, right-hand-side $\Lambda_\bff$ and exponent $\Lambda_p$ provided it is a weak martingale solution in the sense of Definition \ref{def:weak} and the following holds.
\begin{enumerate}
\item We have $\bfF_p(\cdot,\ep(\bfv))\in \lebe^2(0,T;\sobo^{1,2}(\mt))$ $\p$-a.s.,
\item there are $\pi_{\mathrm{det}}$ and $\Phi^\pi$ $(\mf_t)$-progressively measurable
such that $\pi_{\mathrm{det}}\in L^1(Q)$ and $\Phi^\pi\in L^2(0,T;\lebe_2(\mathfrak U;\lebe^2(\mt)))$ $\p$-a.s. as well as
\begin{align}\label{eq:strong}
\begin{aligned}
\bfv(t)&=\bfv(0)+\int_0^t\Big[\Div\Big(\mu(1+|\ep(\bfv)|)^{p(\cdot)-2}\ep(\bfv)\Big)-\Div\big(\bfv\otimes\bfv\big)
-\nabla\pi_{\mathrm{det}}+\bff\Big]\ds\\&+
\int_0^t\big[\Phi(\bfv)+\Phi^\pi\big]\,\dd W
\end{aligned}
\end{align}
$\p$-a.s. for all $t\in[0,T]$.
 \end{enumerate}
\end{definition}
\begin{theorem}[Strong martingale solution]\label{thm:main2}
Let the assumptions of Theorem \ref{thm:main} be satisfied.
Suppose that either we have
\begin{itemize}
\item[(i)] $n=2$ and $1<p^-\leq  p^+< 4$ or;
\item[(ii)] $n=3$ and $\frac{11}{5}<p^-\leq p^+\leq p^-+\frac{4}{5}$.
\end{itemize}
Then there is a strong martingale solution to \eqref{0.4}--\eqref{0.3} in the sense of Definition \ref{def:strong}. We have the estimate
\begin{align}\label{eq:regularityestimate1b}
\begin{split}
\E&\bigg[\sup_{t\in(0,T)}\int_{\mt} |\nabla\bfv(t)|^2\dx+\int_{\Q} |\nabla\bfF_p(\cdot,\ep(\bfv))|^{2}\dxt\bigg]\leq c(\Lambda_0,\Lambda_\bff).
\end{split}
\end{align}
\end{theorem}
\begin{remark}
\noindent
\begin{itemize}
\item We remark that the most interesting situation for physical applications is
when $p$ can vary between 1 and 2 as assumed in part (i) of Theorem \ref{thm:main2}.
This refers to a range between a Newtonian fluid ($p=2$) and a plastic material ($p$ close to 1) which has been observed in experiments on electro-rheological fluids. 
\item Similar to \eqref{thm:main2} ii) it is also possible to gain a result in two dimensions if $p^+\geq4$.
In this case the assumption reads as $p^+<p^-+1$. However this situation is outside the range of physical interest and we leave the details to the reader.
\end{itemize}
\end{remark}

\subsection{\db{Stochastically strong} solutions.}
We are now concerned with the question whether a solution to \eqref{0.4}--\eqref{0.3} can be constructed on a given probability space and a given initial velocity $\bfv_0$ (which is a a random variable rather than a probability law).
This goes hand in hand with the question of unique solvability and holds already on the level of stochastic ODEs  (see, e.g., \cite[Chap. 5]{KS}). We start with a formulation which is weak in the PDE-sense.
\begin{definition}[\db{Weak stochastically strong solution}] \label{def:strsol}
Let $(\Omega,\mf,(\mf_t),\p)$ be a stochastic basis with a complete right-continuous filtration and let ${W}$ be an $(\mf_t) $-cylindrical Wiener process.
Let
$\bfv_0$ be an $\lebe^{2}(\mt)$-valued $\mf_0$-measurable random variable. Let $\bff$ and $p$ be  $(\mf_t)$-progressively measurable processes
such that $\bff\in \lebe^2(\Q)$ and $p\in \hold^{0}([0,T]\times \mt)$ with $p\geq1$ $\p$-a.s..
A function $\bfv$ is called a \db{weak stochastically strong solution} solution to \eqref{0.4}--\eqref{0.3} provided
\begin{enumerate}
\item the velocity field satisfies $\bfv\in C_w([0,T];\lebe^2(\mt))$, $\ep(\bfv)\in \lebe^{p(\cdot)}(Q)$, $\p$-a.s. and is $(\mf_t)$-progressively measurable, 
\item we have $\bfv(0)=\bfv_0$ $\p$-a.s.,
\item for all
 $\bfvarphi\in C^\infty_{\Div}(\mt)$ and all $t\in[0,T]$ there holds $\p$-a.s.
\begin{align*}
\int_{\mt}\bfv(t)\cdot\bfvarphi\dx &+\int_0^t\int_{\mt}\mu(1+|\ep(\bfv)|)^{p(\cdot)-2}\ep(\bfv):\ep(\bfphi)\dxs-\int_0^t\int_{\mt}\bfv\otimes\bfv:\ep(\bfphi)\dxs\\&=\int_{\mt}\bfv(0)\cdot\bfvarphi\dx
+\int_{\mt}\int_0^t\bff\cdot\bfphi\dxs+\int_{\mt}\int_0^t\Phi(\bfv)\,\dd W\cdot \bfvarphi\dx.
\end{align*}
\end{enumerate}
\end{definition}
We obtain the following result (recall Remark \ref{rem:2.3} for the assumptions on $p$ below in \eqref{eq:p-p+pathwise} below).
\begin{theorem}[\db{Weak stochastically strong solution}]\label{thm:main3}
Let
$\bfv_0$ be an $\lebe^{2}(\mt)$-valued $\mf_0$-measurable random variable. Let $\bff$ and $p$ be  $(\mf_t)$-progressively measurable processes such that $\bff\in \lebe^2(\Q)$ and $p\in \hold^{0}([0,T]\times \mt)$ $\p$-a.s. Suppose that
\begin{equation}\label{initial}
\E\big\|\bfv_0\big\|_{\lebe^{2}(\mt)}^{2r}<\infty,\quad
\E\big\|\bff\big\|_{\lebe^2(Q)}^{2r}<\infty,\quad \E\|p\|_{\PPln(Q)}<\infty.
\end{equation}
for all $1\leq r<\infty$ as well as
\begin{equation}\label{initial}
\E\big\|\bfv_0\big\|_{\sobo^{1,2}(\mt)}^2<\infty,\quad
\E\big\|\bff\big\|_{L^2(0,T;\sobo^{1,2}(\mt))}^2<\infty.
\end{equation}
Moreover, assume that we have $\p$-a.s.
\begin{align*}
p^-\leq p\leq p^+,\,\,\|p\|_\infty+\|\nabla p\|_\infty\leq\,c_p,
\end{align*}
where $c_p<\infty$ and
\begin{align}\label{eq:p-p+pathwise}
\frac{n+2}{2}\leq p^-\leq p^+<np^-+4.
\end{align}
Finally, assume that $\Phi$ satisfies \eqref{eq:phi}.
Then there is a \db{weak stochastically strong solution} to \eqref{0.4}--\eqref{0.3} in the sense of Definition \ref{def:strsol}. We have the energy estimate
\begin{align}\label{eq:regularityestimate1c}
\begin{split}
\E&\bigg[\sup_{t\in(0,T)}\int_{\mt} |\bfv(t)|^2\dx+\int_{\Q} |\ep(\bfv)|^{p(\cdot)}\dxt\bigg]^r\\&\leq \,c\,\E\bigg[\int_{\mt} |\bfv_0|^2\dx+\int_{\Q} |\bff|^2\dxt\bigg]^r.
\end{split}
\end{align}
\end{theorem}
\begin{remark}
 As in the deterministic case (see \cite{MNRR} and \cite{Di}) the assumptions
on $p^-$ yielding uniqueness are rather restrictive. The same bounds are needed in Theorem \ref{thm:main3} for the existence of \db{stochastically strong} solutions.
\end{remark}
Having a look at Definitions \ref{def:strong} and \ref{def:strsol} we can expect strong \db{stochastically strong solutions} if the assumptions of Theorems \ref{thm:main2} and \ref{thm:main3} are satisfied. These solutions are strong in both senses.
\begin{definition}[\db{Strong stochastically strong solution}] \label{def:strongstrong}
Let $(\Omega,\mf,(\mf_t),\p)$ be a stochastic basis with a complete right-continuous filtration and let ${W}$ be an $(\mf_t) $-cylindrical Wiener process.
Let
$\bfv_0$ be an $\lebe^{2}(\mt)$-valued $\mf_0$-measurable random variable. Let $\bff$ and $p$ be  $(\mf_t)$-progressively measurable processes such that $\bff\in \lebe^2(\Q)$ and $p\in \hold^{0}([0,T]\times \mt)$ with $p\geq 1$ $\p$-a.s.
A function $\bfv$ is called a \db{strong stochastically strong solution} to \eqref{0.4}--\eqref{0.3} provided it is a \db{weak stochastically strong solution} in the sense of Definition \ref{def:strsol} and the following holds.
\begin{enumerate}
\item We have $\bfF_p(\cdot,\ep(\bfv))\in \lebe^2(0,T;\sobo^{1,2}(\mt))$ $\p$-a.s.,
\item there are $\pi_{\mathrm{det}}$ and $\Phi^\pi$ $(\mf_t)$-progressively measurable such that\\$\pi_{\mathrm{det}}\in \lebe^1(0,T;\sobo^{1,1}(\mt))$ and $\Phi^\pi\in \lebe^2(0,T;\lebe_2(\mathfrak U;\lebe^2(\mt)))$ $\p$-a.s. as well as
\begin{align}\label{eq:strong'}
\begin{aligned}
\bfv(t)&=\bfv(0)+\int_0^t\Big[\Div\Big(\mu(1+|\ep(\bfv)|)^{p(\cdot)-2}\ep(\bfv)\Big)-\Div\big(\bfv\otimes\bfv\big)
-\nabla\pi_{\mathrm{det}}+\bff\Big]\ds\\&+
\int_0^t\big[\Phi(\bfv)+\Phi^\pi\big]\,\dd W
\end{aligned}
\end{align}
$\p$-a.s. for all $t\in[0,T]$.
 \end{enumerate}
\end{definition}
By combining the ideas of the proofs of Theorems \ref{thm:main2} and \ref{thm:main3} we obtain the following corollary (see end of Section \ref{sec:2d} for the proof).
\begin{corollary}
\label{cor:strong}
Let the assumptions of Theorem \ref{thm:main2} be satisfied. Suppose in addition that $p^-\geq\frac{n+2}{2}$.
Then there is a \db{strong stochastically strong solution} to \eqref{0.4}--\eqref{0.3} in the sense of Definition \ref{def:strongstrong}. 
\end{corollary}

\section{Galerkin approximation}
\label{sec:galerkin}
Our approach is a stochastic variant of the usual Galerkin ansatz, thereby reducing the problem of interest to an stochastic ordinary differential equation. In this respect, we firstly record the following fundamental fact on eigenvector expansions for the Stokes operator, the proof of which can be found in the appendix of \cite{MNRR}:
\begin{lemma}
There is a sequence $(\lambda_k)\subset\R$ and a sequence of functions $(\bfw_{k})\subset \sobo_{\Div}^{1,2}(\mt)$ such that the following hold:
\begin{enumerate}
\item For each $k\in\mathbb{N}$, $\bfw_{k}$ is an eigenvector to the eigenvalue $\lambda_k$ of the Stokes--operator in the sense that
\begin{align*}
\langle \bfw_{k},\bfphi\rangle_{\sobo^{1,2}(\mt)}= \lambda_k\int_{\mt}\bfw_k\cdot\bfvarphi\dx\quad\text{for all }\bfvarphi\in \sobo_{\Div}^{1,2}(\mt),
\end{align*}
\item $\int_{\mt}\bfw_k\cdot\bfw_{m}\dif x=\delta_{km}$ for all $k,m\in\mathbb{N}$,
\item $1\leq\lambda_1\leq \lambda_2\leq...$ and $\lambda_k\rightarrow\infty$,
\item $\langle\tfrac{\bfw_k}{\sqrt{\lambda_k}},\tfrac{\bfw_m}{\sqrt{\lambda_m}}\rangle_{\sobo^{1,2}(\mt)}=\delta_{km}$ for all $k,m\in\mathbb{N}$,
\item $(\lambda_k^{-1/2}\bfw_k)$ is a Hilbert space basis of $\sobo_{\Div}^{1,2}(\mt)$.
\end{enumerate}
\end{lemma}
We consider the Skorokhod representation of the law $\Lambda\db{\otimes} \Gamma$, where $\Gamma$ is the law of a cylindrical Wiener process on $\mathfrak U$. We obtain
a probability space $(\Omega,\mf,\p)$, \db{a random variable $(\bfv_0,\bff,p)$ with law $\Lambda$}, as well as a cylindrical Wiener process $W=\sum_k \beta_ke_k$. Finally, we set
$$\mf_t:=\sigma\Big(\db{\sigma[\bfv_0]}\cup\sigma[ p\db{|_{[0,t]}}]\cup \sigma[\bff\db{|_{[0,t]}}]\cup\bigcup_{k=1}^\infty\sigma[ W_k\db{|_{[0,t]}}]\cup\db{\{N\in\mf:\p(N)=0\}}\Big),\quad t\in[0,T].$$
Our objective for the rest of the section is to establish the existence of solutions $\vv^{N}$ of the system \eqref{0.1} in the particular form
\begin{align}
\bfv^N=\sum_{k=1}^N c_k^N\bfw_k=\bfC^N\cdot\bfomega^N,\quad \bfomega^N=(\bfw_1,...,\bfw_N),
\end{align}
where $\bfC^N=(c_i^N):\Omega\times (0,T)\rightarrow \R^N$. Our aim is hereafter to solve ($k=1,...,N$)
\begin{align}
\begin{aligned}
\int_{\mt} &\dd\bfv^N\cdot\bfw_k\dx+\int_{\mt}\bfS_p(\cdot,\ep(\bfv^N)):\ep(\bfw_k)\dx\dt\\ & =\int_{\mt}\bfv^N\otimes\bfv^N:\nabla\bfw_k\dx\dt+\int_{\mt}\bff\cdot\bfw_k\dx\dt+\int_{\mt} \Phi(\bfv^N)\,\dd W^N\cdot\bfw_k\dx,
\\
&\bfv^N(0)=\mathcal P^N\bfv_0.\label{eq:gal}
\end{aligned}
\end{align}
with
\begin{align*} 
  \bfS_p(\omega,t,x,\bfeta)=\mu(1+|\bfeta|)^{p(\omega,t,x)-2}\bfeta.
\end{align*}
Here $\mathcal P^N:\lebe^{2}_{\Div}(\mt)\rightarrow \mathcal X_N:=\mathrm{span}\left\{\bfw_1,...,\bfw_N\right\}$ is the orthogonal projection, i.e.
\begin{align*}
\mathcal P^N\bfu=\sum_{k=1}^N\langle \bfu,\bfw_k\rangle_{\lebe^{2}}\bfw_k.
\end{align*} 
The equation above is to be understood $\mathbb P$ a.s. and for a.e. $t$ and we set
\begin{align*}
W^N=\sum_{k=1}^N\bfe_k \beta_k=\bfe^N\cdot \bfbeta^N.
\end{align*}
It is equivalent to solving
\begin{align}\label{SDE*}
\begin{cases}\dd\bfC^N&=\big[\bfmu (t,\bfC^N)\big]\dt+\bfSigma (\bfC^N)\,\dd\bfbeta^N_t\\
\bfC^N(0)&=\bfC_0\end{cases}
\end{align}
with the abbreviations
\begin{align*}
\bfmu (\bfC^N)&=\bigg(-\int_{\mt}\bfS_p(\cdot,\bfC^N\cdot\ep( \bfw^N)):\ep(\bfw_k)\dx+\int_{\mt}(\bfC^N\cdot\bfw^N)\otimes (\bfC^N\cdot\bfw^N):\nabla\bfw_k\dx\bigg)_{k=1}^N\\
&+\bigg(\int_{\mt}\bff(t)\cdot\bfw_k\dx\bigg)_{k=1}^N,\\
\bfSigma(\bfC^N)&=\bigg(\int_\mt \Phi(\bfC^N\cdot W^N)\bfe_l\cdot \bfw_k\dx\bigg)_{k,l=1}^N,\\
\bfC_0&=\Big(\langle\bfv_0,\bfw_k\rangle_{\lebe^{2}(\mt)}\Big)_{k=1}^N.
\end{align*}
We apply the results from \cite{PrRo}, Thm. 3.1.1. In the following we will check the assumptions. We have by the monotonicity of $\bfS_p$
\begin{align*}
\big(\bfmu(t,\bfC^N)&-\bfmu(t,\tilde{\bfC}^N)\big)\cdot\big(\bfC^N-\tilde{\bfC}^N\big)\\
&=-\int_{\mt}\big(\bfS_p(\cdot,\ep(\bfv^N))-\bfS_p(\cdot,\ep(\tilde{\bfv}^N))\big):\big(\ep(\bfv^N)-\ep(\tilde{\bfv}^N)\big)\dx\\
&+\int_{\mt}\big(\bfv^N\otimes \bfv^N-\tilde{\bfv}^N\otimes \tilde{\bfv}^N\big):\big(\ep(\bfv^N)-\ep(\tilde{\bfv}^N)\big)\dx\\
&\leq \int_{\mt}\big(\bfv^N\otimes \bfv^N-\tilde{\bfv}^N\otimes \tilde{\bfv}^N\big):\big(\ep(\bfv^N)-\ep(\tilde{\bfv}^N)\big)\dx.
\end{align*}
If $|\bfC^N|\leq R$ and $|\tilde{\bfC}^N|\leq R$ there holds
\begin{align*}
\big(\bfmu(t,\bfC^N)&-\bfmu(t,\tilde{\bfC}^N)\big)\cdot\big(\bfC^N-\tilde{\bfC}^N\big)
\leq c(R,N)|\bfC^N-\tilde{\bfC}^N|^2.
\end{align*}
Here we took into account boundedness of $\bfw_k$ and $\nabla\bfw_k$.
This implies weak monotonicity in the sense of \cite{PrRo}, (3.1.3) using Lipschitz continuity $\bfSigma$ in $\bfC^N$, cp. (\ref{eq:phi}).
On account of $\int_{\mt} \bfv^N\otimes\bfv^N:\ep(\bfv^N)\dx=0$ there holds
further
\begin{align*}
\bfmu(t,\bfC^N)\cdot\bfC^N&=-\int_{\mt}\bfS_p(\cdot,\ep(\bfv^N)):(\ep(\bfv^N)\dx+\int_{\mt}\bff(t)\cdot\bfv^N\dx\leq c\,(1+\|\bff(t)\|_2\|\bfv^N\|_2)\\
&\leq (1+\|\bff(t)\|_2)(1+\|\bfv^N\|^2)\leq \,c\,(1+\|\bff(t)\|_2)(1+|\bfC^N|^2).
\end{align*}
So we have using the linear growth of $\bfSigma$ which follows from \ref{eq:phi}
\begin{align*}
\bfmu(\bfC^N)\cdot\bfC^N+|\bfSigma(\bfC^N)|^2\leq c(+\|\bfv^N\|_2^2)\big(1+|\bfC^N|^2\big).
\end{align*}
As the integral $\int_0^T (1+\|\bff(t)\|_2)\dt$ is finite $\p$-a.s. this yields weak coercivity in the sense of \cite{PrRo}, (3.1.4). We obtain a unique strong solution $\bfC^N\in \lebe^2(\Omega;C[0,T])$ to the SDE (\ref{SDE*}).\\\

We obtain the following a priori estimate.

\begin{theorem}\label{thm:2.1'}
Assume (\ref{0.3}) with $p:\Omega\times Q\rightarrow(1,\infty)$, (\ref{eq:phi}) and
for some $r\geq1$
\begin{equation}\label{initial'}
\int_{\lebe^2_{\Div}(\mt)}\big\|\bfu\big\|_{\lebe^2(\mt)}^{2r}\,\dd\Lambda_0(\bfu)<\infty,\quad
\int_{\lebe^2(\Q)}\big\|\bfg\big\|_{\lebe^2(\Q)}^{2r}\,\dd\Lambda_\bff(\bfg)<\infty.
\end{equation}
Then there holds uniformly in $N$
\begin{align}\label{eq:regularityestimate1}
\begin{split}
\E&\bigg[\sup_{t\in(0,T)}\int_{\mt} |\bfv^N(t)|^2\dx+\int_{\Q} |\ep(\bfv^N)|^{p(\cdot)}\dxt\bigg]^r\leq C_r(\Lambda_0,\Lambda_\bff),\\
C_r(\Lambda_0,\Lambda_\bff)&=c\,\bigg(1+
\int_{\lebe^2_{\Div}(\mt)}\big\|\bfu\big\|_{\lebe^2(\mt)}^{2r}\,\dd\Lambda_0(\bfu)+
\int_{\lebe^2(\Q)}\big\|\bfg\big\|_{\lebe^2(\Q)}^{2r}\,\dd\Lambda_\bff(\bfg)\bigg),
\end{split}
\end{align}
provided $C_r(\Lambda_0,\Lambda_\bff)$ is finite.
\end{theorem}

\begin{proof}
We apply It\^{o}'s formula to the function $f(\bfC)=\tfrac{1}{2}|\bfC|^2$ which shows
\begin{align*}
\frac{1}{2}\|\bfv^N(t)\|_{\lebe^2(\mt)}^2
&=\frac{1}{2}\|\bfC^N(0)\|_{\lebe^2(\mt)}^2+\sum_{k=1}^N\int_0^t  c_k^N\,\dd(c_k^N)+\frac{1}{2}\sum_{k=1}^N\int_0^t \,\dd\langle\langle c_k^N\rangle\rangle\\
&=\frac{1}{2}\|\mathcal P^N\bfv_0\|_{\lebe^2(\mt)}^2-\int_0^t\int_{\mt}\bfS_p(\cdot,\ep(\bfv^N)):\ep(\bfv^N)\dxs\\
&+\int_0^t\int_{\mt}\bff\cdot\bfv^N\dxs+\int_{\mt}\int_0^t\bfv^N\cdot \Phi(\bfv^N)\,\dd\sobo^N\dx\\
&+\frac{1}{2}\int_{\mt}\int_0^t\,\dd\Big\langle\Big\langle\int_0^\cdot\Phi(\bfv^N)\,\dd\sobo^N\Big\rangle\Big\rangle\dx.
\end{align*}
Here we used $\dd\bfv^N=\sum_{k=1}^N \dd c_k^N\bfw_k$, $\int_{\mt}\bfv^N\otimes\bfv^N:\nabla\bfv^N\dx=0$ and property (ii) of the base $(\bfw_k)$.
Now we can follow, taking the $r$-th power, and the supremum, building expectations and using (\ref{0.3}) that
\begin{align*}
\E\bigg[\sup_{(0,T)}\int_{\mt} &|\bfv^N(t)|^2\dx+\int_0^T\int_{\mt} |\ep(\bfv^N)|^{p(\cdot)}\dxs\bigg]^r\\&\leq \,c\,\E\bigg[1+\|\bfv_0\|^2_{\lebe^2(\mt)}+J_1(T)+\sup_{(0,T)}J_2(t)+J_3(T)\bigg]^r.
\end{align*}
Here we abbreviated
\begin{align*}
J_1(t)&=\int_0^t\int_{\mt}|\bff||\bfv^N|\dxs,\\
J_2(t)&=\int_{\mt}\int_0^t\bfv^N\cdot \Phi(\bfv^N)\,\dd W^N\dx,\\
J_3(t)&=\int_{\mt}\int_0^t\,\dd\Big\langle\Big\langle\int_0^\cdot\Phi(\bfv^N)\,\dd W^N\Big\rangle\Big\rangle\dx.
\end{align*}
We obviously have
\begin{align*}
J_1\leq\,\int_0^t\int_{\mt}|\bff|^2\dxs+\int_0^t\int_{\mt}|\bfv^N|^2\dxs.
\end{align*}
Straightforward calculations show on account of (\ref{eq:phi})
\begin{align*}
\E[J_3]^r&=\E\bigg[\sum_{k=1}^N\int_0^t\bigg(\int_{\mt}  \Phi(\bfv^N)\bfe_k\dx\bigg)^2\ds\bigg]^r\\
&\leq\E\bigg[\sum_{k=1}^\infty\int_0^t\int_{\mt}  |g_k(\bfv^N)|^2\dxs\bigg]^r\\
&\leq\,c\,\E\bigg[1+\int_0^t\int_{\mt} |\bfv^N|^2\dxs\bigg]^r.
\end{align*}
On account of Burgholder-Davis-Gundi inequality, Young's inequality and (\ref{eq:phi}) we gain
\begin{align*}
\E\bigg[\sup_{t\in(0,T)}|J_2(t)|\bigg]^r&=\E\bigg[\sup_{t\in(0,T)}\bigg|\int_0^t\int_{\mt}\bfv^N\cdot\Phi(\bfv^N)\dx\,\dd W^N\bigg|\bigg]^r\\
&=\E\bigg[\sup_{t\in(0,T)}\bigg|\int_0^t\sum_{k=1}^N\int_{\mt}\bfv^N\cdot g_k(\bfv^N)\dx\,\dd\beta_k\bigg|\bigg]^{r}\\
&\leq c\,\E\bigg[\int_0^T\sum_{k=1}^N\bigg(\int_{\mt}\bfv^N\cdot g_k(\bfv^N)\dx\bigg)^2\dt\bigg]^{\frac{1}{2}}\\
&\leq c\,\E\bigg[\bigg(\int_0^T\bigg(\sum_{k=1}^\infty \int_{\mt}|\bfv^N|^2\dx\int_{\mt} |g_k(\bfv^N)|^2\dx\bigg)\dt\bigg]^{\frac{r}{2}}\\
&\leq c\,\E\bigg[1+\int_0^T\bigg( \int_G|\bfv^N|^2\dx\bigg) ^2\dt\bigg]^{\frac{r}{2}}\\
&\leq \delta\,\E\bigg[\sup_{t\in(0,T)}\int_G|\bfv^N|^2\dx\bigg]^r+c(\delta)\,\E\bigg[1+\int_0^T\int_G|\bfv^N|^2\dxt\bigg]^r,
\end{align*}
where $\delta>0$ is arbitrary.
This finally proves the claim by Gronwall's lemma for $\delta$ sufficiently small using $\Lambda_0=\p\circ\bfv_0^{-1}$ and $\Lambda_\bff=\p\circ\bff^{-1}$.
\end{proof}

\section{Analytically weak solutions}
\label{sec:weak}
This section is devoted to the proof of Theorems \ref{thm:main} and \ref{thm:main3}. In view of compactness, our main concern is the derivation of fractional estimates for $\nabla\bfv^N$. Based on this we are able to apply the stochastic compactness method employing Skorokhod's theorem to pass to the limit in the Galerkin approximation from the previous section.

\subsection{Fractional differentiability}
To set up fractional estimates in a convenient manner, we introduce the concave function for $\theta\geq0$
\begin{align*}
g(\theta)=g_{\lambda}(\theta):=\begin{cases}\frac{1}{1-\lambda}(1+\theta)^{1-\lambda},\quad &\lambda\neq1\\
\ln(1+\theta),\quad &\lambda=1
\end{cases}
\end{align*}
for 
\begin{align*}
\lambda=\tfrac{2(\overline q-p^{-})}{np^{-}-\overline qn+4},
\end{align*}
where $\overline q=\max\{3,p^++\varrho\}$ with $\varrho>0$ arbitrarily small. The additional power $\varrho$ arises from the elementary inequality
\begin{align}\label{eq:varrho}
\ln(1+|\bfxi|)\leq c_\varrho (1+|\bfxi|^\varrho)\quad\bfxi\in \R^{n\times n}.
\end{align}
Note that the denominator in the definition of $\lambda$ is positive as long as
\begin{align}\label{eq:p_new}
p^->\frac{\overline qn-4}{n}.
\end{align}
Similar to \cite[section 3]{TeYo} we have the following theorem.
\begin{theorem}\label{thm:2.1'''}
Suppose that
\begin{equation}\label{initial}
\int_{\lebe^{2}_{\Div}(\mt)}\big\|\bfu\big\|_{\sobo^{1,2}(\mt)}^2\,\dd\Lambda_0(\bfu)<\infty,\quad
\int_{\lebe^2(\Q)}\big\|\bfg\big\|_{L^2(0,T;\sobo^{1,2}(\mt))}^2\,\dd\Lambda_\bff(\bfg)<\infty.
\end{equation}
Moreover, assume that $\p$-a.s. $p\in \hold^0([0,T]\times\mt)$ such that $\p$-a.s. we have
\begin{align}\label{eq:lawp'}
1<p^-\leq p\leq p^+,\,\,\|\nabla p\|_\infty\leq\,c_p,
\end{align}
where $c_p<\infty$ and that \eqref{eq:p_new} holds.
Finally, assume that $\Phi$ satisfies \eqref{eq:phi}. Then we have
\begin{itemize}
\item[a)] If $p^{-}\geq 2$ then there holds uniformly in $N$:
\begin{align*}
\E&\bigg[\int_0^T\frac{\|\nabla^2\bfv^N(t)\|_2^2}{(1+\|\nabla\bfv^N(t)\|_2^2)^\lambda}\dt\bigg]\leq\, C_1(\Lambda_0,\Lambda_\bff).
\end{align*}
\item[a)] If $p^{-}< 2$ then there holds uniformly in $N$:
\begin{align*}
\E&\bigg[\int_0^T\frac{\|\nabla^2\bfv^N(t)\|_{p^{-}}^2}{(1+\|\nabla\bfv^N(t)\|_2^2)^\lambda(1+\|\nabla\bfv^N(t)\|_{p^{-}})^{2-p^{-}}}\dt\bigg]\leq\,
C_1(\Lambda_0,\Lambda_\bff).
\end{align*}
\end{itemize}
\end{theorem}
\begin{proof}
%
We start with the evolution of $\|\nabla\bfv^N(t)\|_{\lebe^2(\mt)}^2$. Applying It\^{o}'s formula to the mapping $ \bfC\mapsto \|\nabla\bfv\|_2^2$, where $\bfC=(c^1,...,c^N)$ and $\bfv$ are related through $\bfv=\sum_{k=1}^N c_k\bfw_k$.
We obtain
\begin{align*}
& \frac{1}{2}\|\nabla\bfv^N(t)\|_{\lebe^2(\mt)}^2
=\frac{1}{2}\|\nabla\mathcal P^N\bfv_0\|_{\lebe^2(\mt)}^2-\int_0^t\int_{\mt} D_\bfxi\bfS(\cdot,\ep(\bfv^N))(\partial_{\gamma}\ep(\bfv^N),\partial_{\gamma}\ep(\bfv^N))\dxs\\&-\int_0^t\int_{\mt} D_x\bfS(\cdot,\ep(\bfv^N)):\partial_{\gamma}\nabla\bfv^N\dxs+\int_0^t\int_{\mt} \Div\big(\bfv^N\otimes\bfv^N\big):\Delta\bfv^N\dxs\\
&+\int_0^t\int_{\mt}\partial_{\gamma}\bfv^N
\cdot\partial_{\gamma}\Big(\Phi(\bfv^N)\,\dd W\Big)\dx+\frac{1}{2}\int_{\mt}\int_0^t\dd\Big\langle\Big\langle\int_0^{\cdot}\partial_{\gamma} \big(\Phi(\bfv^N)\,\dd W\big)\Big\rangle\Big\rangle\dx,
\end{align*}
where the sum is taken over all $\gamma\in\{1,\dots,n\}$. 
Now we apply It\^{o}'s formula
to the mapping $ \bfC^N\mapsto g_\lambda (\|\nabla\bfv\|_2^2)$ and obtain
\begin{align*}
g_\lambda (\|\nabla\bfv^N(t)\|_2^2)&=g_\lambda (\|\nabla\bfv^N(0)\|_2^2)+\int_0^t\frac{1}{(1+\|\nabla\bfv^N\|^2_2)^\lambda}\dd \|\nabla\bfv^N\|_2^2\\&-\frac{\lambda}{2}\int_0^t\frac{1}{(1+\|\nabla\bfv^N\|^2_2)^{\lambda+1}}\dd \big\langle\big\langle\|\nabla\bfv^N\|_2^2\big\rangle\big\rangle,
\end{align*}
where we have
\begin{align*}
\int_0^t&\frac{2}{(1+\|\nabla\bfv^N\|^2_2)^\lambda}\dd \|\nabla\bfv^N\|_2^2\\=&-\int_0^t\frac{1}{(1+\|\nabla\bfv^N\|^2_2)^\lambda}\int_{\mt} D_\bfxi\bfS(\cdot,\ep(\bfv^N))(\partial_{\gamma}\ep(\bfv^N),\partial_{\gamma}\ep(\bfv^N))\dxs\\&-\int_0^t\frac{2}{(1+\|\nabla\bfv^N\|^2_2)^\lambda}\int_{\mt} D_x\bfS(\cdot,\ep(\bfv^N)):\partial_{\gamma}\nabla\bfv^N\dxs\\&+\int_0^t\frac{2}{(1+\|\nabla\bfv^N\|^2_2)^\lambda}\int_{\mt} \Div\big(\bfv^N\otimes\bfv^N\big):\Delta\bfv^N\dxs\\
&+\int_0^t\frac{2}{(1+\|\nabla\bfv^N\|^2_2)^\lambda}\int_{\mt}\partial_{\gamma}\bfv^N
\cdot\partial_{\gamma}\Big(\Phi(\bfv^N)\,\dd W\Big)\dx\\&+\int_0^t\int_{\mt}\frac{1}{(1+\|\nabla\bfv^N\|^2_2)^\lambda}\dd\Big\langle\Big\langle\int_0^{\cdot}\partial_{\gamma} \big(\Phi(\bfv^N)\,\dd W\big)\Big\rangle\Big\rangle\dx\\
&=-J_1-J_2+J_3+J_4+J_5.
\end{align*}
Moreover, there holds
\begin{align*}
-\frac{\lambda}{2}&\int_0^t\frac{1}{(1+\|\nabla\bfv^N\|^2_2)^{\lambda+1}}\dd \big\langle\big\langle\|\nabla\bfv^N\|_2^2\big\rangle\big\rangle\leq 0.
\end{align*}
$\p$-a.s. such that this term can be neglected. We start with the lower estimate
\begin{align*}
J_1
&\geq c\int_0^t\frac{1}{(1+\|\nabla\bfv^N\|^2_2)^\lambda}\int_{\mt}(1+|\ep(\bfv^N)|)^{p(\cdot)-2}
|\nabla\ep(\bfv^N)|^2\dxs\\
&\geq c\int_0^t\frac{1}{(1+\|\nabla\bfv^N\|^2_2)^\lambda}\int_{\mt}(1+|\ep(\bfv^N)|)^{p^{-}-2}
|\nabla\ep(\bfv^N)|^2\dxs.
\end{align*}
All other terms will be estimate form above. By Young's inequality we obtain
using \eqref{eq:varrho}
\begin{align*}
J_2
&\leq c\int_0^t\frac{1}{(1+\|\nabla\bfv^N\|^2_2)^\lambda}\int_{\mt}\ln(1+|\ep(\bfv^N)|)(1+|\ep(\bfv^N)|)^{p(\cdot)-1}
|\nabla\ep(\bfv^N)|\dxs\\
&\leq \kappa\int_0^t\frac{1}{(1+\|\nabla\bfv^N\|^2_2)^\lambda}\int_{\mt}(1+|\ep(\bfv^N)|)^{p(\cdot)-2}
|\nabla\ep(\bfv^N)|^2\dxs\\&+c(\kappa)\int_0^t\frac{1}{(1+\|\nabla\bfv^N\|^2_2)^\lambda}\int_{\mt}\big(1+|\nabla\bfv^N|^{\q}\big)\dxt,
\end{align*}
where $\kappa>0$ is arbitrary. For $\kappa$ small enough we will be able to absorb the corresponding term in $J_1$.
Moreover, we have
\begin{align*}
J_3&\leq \int_0^t\frac{1}{(1+\|\nabla\bfv^N\|^2_2)^\lambda} \int_{\mt}|\nabla\bfv^N|^3\dxs\\
&\leq \int_0^t\frac{1}{(1+\|\nabla\bfv^N\|^2_2)^\lambda} \int_{\mt}\big(1+|\nabla\bfv^N|^{\q}\big)\dxs
\end{align*}
using integration by parts. Finally, we obtain from \eqref{eq:phi}
\begin{align*}
J_5&=\sum_k\int_0^t\frac{1}{(1+\|\nabla\bfv^N\|^2_2)^\lambda}\bigg(\int_{\mt}\nabla g_k(\bfv^N)\dx\bigg)^2\dt\\
&\leq\,\sum_k\int_0^t\frac{1}{(1+\|\nabla\bfv^N\|^2_2)^\lambda}\int_{\mt}|\nabla g_k(\bfv^N)|^2\dx\dt\\
&\leq\,c\,\int_0^t\frac{1}{(1+\|\nabla\bfv^N\|^2_2)^\lambda}\int_{\mt}|\nabla\bfv^N|^2\dx\dt\\
&\leq\,c\,\int_0^t\frac{1}{(1+\|\nabla\bfv^N\|^2_2)^\lambda}\int_{\mt}\big(1+|\nabla\bfv^N|^{\q}\big)\dx\dt.
\end{align*}
Applying expectations (note that $\E[J_4]=0$) and choosing $\kappa$ small enough
we end up with
\begin{align}
\nonumber\E g_\lambda(\|\nabla\bfv^N(t)\|_2^2)&+\E\int_0^t\frac{1}{(1+\|\nabla\bfv^N\|^2_2)^\lambda}\int_{\mt}(1+|\ep(\bfv^N)|)^{p(\cdot)-2}|\nabla\ep(\bfv^N)|^2\dxs\\
\label{eq:final0}&\leq\,c\,\E\bigg[ g_\lambda(\|\nabla\bfv^N(0)\|_2^2)+\int_0^t\frac{1}{(1+\|\nabla\bfv^N\|^2_2)^\lambda}\int_{\mt}\big(1+|\nabla\bfv^N|^{\q}\big)\dxt\bigg].
\end{align}
The last term on the right-hand side cannot be controlled so far. In order to suitably bound $\|\nabla \vv^{N}\|_{\q}^{\q}$, let $2> q\geq n(\overline q-p^{-})/\overline q$, existence of which follows from \eqref{eq:p_new} and $\overline q>2$, and put
\begin{align}\label{eq:defalpha}
\alpha:=\frac{p^{-}(np^{-}+2q-\q n)}{2(np^{-}+\q q-\q n)}\;\;\text{so that}\;\;1-\alpha=\frac{(\q-p^{-})(np^{-}+2q-2n)}{2(np^{-}+\q q-\q n)}
\end{align}
so that, in particular, $np^{-}/(n-q)\geq \q$. By Lyapunov's interpolation inequality, we obtain 
\begin{align}
\begin{split}
\|\nabla\vv^{N}\|_{\q}\leq \|\nabla\vv^{N}\|_{2}^{\theta_{1}}\|\nabla\vv^{N}\|_{np^{-}/(n-q)}^{\theta_{2}}\\
\|\nabla\vv^{N}\|_{\q}\leq \|\nabla\vv^{N}\|_{p^-}^{\theta_{3}}\|\nabla\vv^{N}\|_{np^{-}/(n-q)}^{\theta_{4}},
\end{split}
\end{align}
where 
\begin{align*}
\theta_{1}:=\frac{2(np^{-}+\q q-\q n)}{\q(np^{-}+2q-2n)},\;\theta_{2}:=\frac{(\q-2)np^{-}}{\q(np^{-}+2q-2n)},\;\theta_{3}:=\frac{np^{-}+ \q q-\q n}{ \q q},\;\theta_{4}:=\frac{n(\q-p^{-})}{\q q}. 
\end{align*}
We then obtain 
\begin{align}\label{eq:intermediate}
\begin{split}
\|\nabla\vv^{N}\|_{\q}^{\q} & = \|\nabla \vv^{N}\|_{\q}^{\q(1-\alpha)}\|\nabla\vv^{N}\|_{\q}^{\q\alpha}\\
& \leq \|\nabla\vv^{N}\|_{2}^{\q(1-\alpha)\theta_{1}}\|\nabla\vv^{N}\|_{np^{-}/(n-q)}^{\q(1-\alpha)\theta_{2}+\q\alpha\theta_{4}}(1+\|\nabla\vv^{N}\|_{p^{-}})^{\q\alpha\theta_{3}}\\
& = \|\nabla\vv^{N}\|_{2}^{2q_{1}}(1+\|\nabla\vv^{N}\|_{p^{-}})^{q_{2}}(\|\nabla\vv^{N}\|_{np^{-}/(n-q)})^{q_{3}} =(*), 
\end{split}
\end{align}
where $q_{1},q_{2},q_{3}$ are defined in the obvious manner. To estimate $(*)$, we note that for $\mathbb{P}\otimes\mathscr{L}^{1}$-a.e. $(\omega,t)\in\Omega\times [0,T]$ there holds by Korn's inequality.
\begin{align*}
\|\nabla\uu(\omega,t,\cdot)\|_{\frac{np^{-}}{n-q}} & \leq  \,c\|\ep(\uu(\omega,t,\cdot))\|_{\frac{np^{-}}{n-q}}
\end{align*}

Next we claim that there exists a constant $C>0$ independent of $N\in\mathbb{N}$ such that 
\begin{align}\label{eq:franzwoistderbeweis?}
\begin{split}
\|\nabla\vv^{N}\|_{np^{-}/(n-q)}&\leq C\left(\int_{\mt} D_\bfxi\bfS(\cdot,\ep(\bfv^N))(\partial_{\gamma}\ep(\bfv^N),\partial_{\gamma}\ep(\bfv^N))\dx \right)^{\frac{q}{2p^{-}}}\times \\
& \times \big(1+\|\nabla\vv^{N}\|_{p^{-}})^{\frac{2-q}{2}}
\end{split}
\end{align}
holds $\p$--a.e. in $\Omega$. The estimate \eqref{eq:franzwoistderbeweis?}
is a consequence of the interpolation of $L^{\frac{np^-}{n-q}}(\mt)$ between $L^{p^-}(\mt)$ and $L^{\frac{np^-}{n-2}}(\mt)$, Sobolev's embedding $W^{1,2}(\mt)\hookrightarrow L^{\frac{2n}{n-2}}(\mt)$ (if $n=2$ we have to replace $\frac{n}{n-2}$ by an arbitrary finite exponent)
and the inequality
\begin{align*}
\big|\nabla (1+|\ep(\bfv^N)|)^{\frac{p^-}{2}}\big|^2&\leq\,c(1+|\ep(\bfv^N)|)^{\frac{p^--2}{2}}|\nabla\ep(\bfv^N)|^2\\
&\leq\,c D_\bfxi\bfS(\cdot,\ep(\bfv^N))(\partial_{\gamma}\ep(\bfv^N),\partial_{\gamma}\ep(\bfv^N)).
\end{align*}
Using \eqref{eq:franzwoistderbeweis?}, we further estimate \eqref{eq:intermediate} by use of Young's inequality for any $r>1$ and $\kappa>0$
\begin{align*}
(*) & \leq C\|\nabla\vv^{N}\|_{2}^{\q(1-\alpha)\theta_{1}}\big(1+\|\nabla\vv^{N}\|_{p^{-}})^{\frac{2-q}{2}(\q(1-\alpha)\theta_{2}+\q\alpha\theta_{4})+\q\alpha\theta_{3}}\\ & \times \left(\int_{\mt} D_\bfxi\bfS(\cdot,\ep(\bfv^N))(\partial_{\gamma}\ep(\bfv^N),\partial_{\gamma}\ep(\bfv^N))\dx \right)^{\frac{q}{2p^{-}}(\q(1-\alpha)\theta_{2}+\q\alpha\theta_{4})}\\
& \leq C(\kappa,r)\Big(\|\nabla\vv^{N}\|_{2}^{\q(1-\alpha)\theta_{1}}\big(1+\|\nabla\vv^{N}\|_{p^{-}})^{\frac{2-q}{2}(\q(1-\alpha)\theta_{2}+\q\alpha\theta_{4})+\q\alpha\theta_{3}}\Big)^{\frac{r}{r-1}}\\
& + \kappa\left(\int_{\mt} D_\bfxi\bfS(\cdot,\ep(\bfv^N))(\partial_{\gamma}\ep(\bfv^N),\partial_{\gamma}\ep(\bfv^N))\dx \right)^{\frac{q}{2p^{-}}(\q(1-\alpha)\theta_{2}+\q\alpha\theta_{4})r}.
\end{align*}
To determine the relevant parameters, we shall now require 
\begin{align}\label{eq:exprequirements}
\frac{q}{2p^{-}}(\q(1-\alpha)\theta_{2}+\q\alpha\theta_{4})r=1,\;\;\;\; \Big(\frac{2-q}{2}(\q(1-\alpha)\theta_{2}+\q\alpha\theta_{4})+\q\alpha\theta_{3}\Big)\frac{r}{r-1}=p^-.
\end{align}
Indeed \eqref{eq:exprequirements} is satisfied indeed provided $\alpha$ is defined by \eqref{eq:defalpha} and we have
\begin{align}\label{eq:determiner}
r=\frac{4}{\q n-np^{-}}\;\;\;\text{und}\;\;\;r'=\frac{4}{np^{-}-\q n+4}.
\end{align}
On the other hand, this implies 
\begin{align*}
\q(1-\alpha)\theta_{1}r' 
& = \frac{4(\q-p^{-})}{np^{-}-\q n+4}.
\end{align*}
We obtain 
\begin{align*}
(*) & \leq C(\kappa,r)\|\nabla\vv^{N}\|_{2}^{\q(1-\alpha)\theta_{1}\frac{r}{r-1}}\big(1+\|\nabla\vv^{N}\|_{p^{-}})^{p^{-}}\\
& + \kappa\int_{\mt} D_\bfxi\bfS(\cdot,\ep(\bfv^N))(\partial_{\gamma}\ep(\bfv^N),\partial_{\gamma}\ep(\bfv^N))\dx.
\end{align*}
Inserting this into \eqref{eq:final0}, choosing $\kappa$ small enough can recalling the definition of $\lambda$ yields by Korn's inequality
\begin{align}\label{eq:2211}
\begin{aligned}
\E g_\lambda(\|\nabla\bfv^N(t)\|_2^2)&+\E\int_0^t\frac{1}{(1+\|\nabla\bfv^N\|^2_2)^\lambda}\int_{\mt}(1+|\ep(\bfv^N)|)^{p(\cdot)-2}|\nabla\ep(\bfv^N)|^2\dxs\\
&\leq\,c\,\E\bigg[ g_\lambda(\|\nabla\bfv_0\|_2^2)+\int_0^t\int_{\mt}\big(1+|\nabla\bfv^N|^{p^-}\big)\dxt\bigg]\\
&\leq\,c\,\E\bigg[ g_\lambda(\|\nabla\bfv_0\|_2^2)+\int_0^t\int_{\mt}\big(1+|\ep(\bfv^N)|^{p^-}\big)\dxt\bigg]\\
&\leq\,c\,\E\bigg[ g_\lambda(\|\nabla\bfv_0\|_2^2)+\int_0^t\int_{\mt}\big(1+|\ep(\bfv^N)|^{p(\cdot)}\big)\dxt\bigg]
\end{aligned}
\end{align}
where the right-hand side is uniformly bounded by $C_1(\Lambda_0,\Lambda_\bff)$,
cp. Theorem \ref{thm:2.1'}.
 If $p^{-}\geq2$ the claim follows directly by Korn's inequality. 
If $p^-<2$ we estimate using again Korn's inequality
\begin{align*}
\|\nabla^2\bfv^N(t)\|_{p^{-}}^2&\leq\,c\bigg(\int_{\mt}|\nabla\ep(\bfv^N)|^{p^-}\dx\bigg)^{\frac{2}{p^-}}\\
&=\,c\,\bigg(\int_{\mt}(1+|\ep(\bfv^N)|)^{p^-\frac{p^--2}{2}}|\nabla\ep(\bfv^N)|^{p^-}(1+|\ep(\bfv^N)|)^{p^-\frac{2-p^-}{2}}\dx\bigg)^{\frac{2}{p^-}}\\
&\leq\,c\,\int_{\mt}(1+|\ep(\bfv^N)|)^{p^--2}|\nabla\ep(\bfv^N)|^{2}\dx\bigg(\int_{\mt}(1+|\nabla\bfv^N|)^{p^-}\dx\bigg)^{\frac{2-p^-}{p^-}}.
\end{align*}
So, the claim follows again from \eqref{eq:2211} and $p^-\leq p$.
\end{proof}

\begin{corollary}
\label{cor:frac}
Let assumptions of Theorem \ref{thm:2.1'''} be satisfied. Assume in addition that $p^->\frac{\q n}{n+2}$ if $p^-<2$. Then for any $\overline p<\min\{p^-,\frac{2n}{n-2}\}$ there is $\beta>0$ such that
\begin{align*}
\E&\bigg[\int_0^T\|\nabla\bfv^N\|^{ \overline p}_{\beta,\overline p}\dt\bigg]\leq \,C_1(\lambda_0,\Lambda_\bff)
\end{align*}
uniformly in $N$.
\end{corollary}
\begin{proof}
If $p^-<2$ we set (recall that $p^->\frac{\q n}{n+2}$)
\begin{align*}
 \beta=\frac{((n+2)p^--\q n)p^-}{2((n+5)p^--\q n-(p^-)^2)}\in\Big(0,\frac{1}{2}\Big)
\end{align*}
and obtain
\begin{align}
\nonumber\E\bigg[\int_0^T\|\nabla^2\bfv^N(t)\|_{p^{-}}^{2\beta}\bigg]&=\E\bigg[\int_0^T\Big((1+\|\nabla\bfv^N(t)\|_2^2)^\lambda(1+\|\nabla\bfv^N(t)\|_{p^{-}})^{2-p^{-}}\Big)^\beta\\
\nonumber&\times\Big(\frac{\|\nabla^2\bfv^N(t)\|_{p^{-}}^2}{(1+\|\nabla\bfv^N(t)\|_2^2)^\lambda(1+\|\nabla\bfv^N(t)\|_{p^{-}})^{2-p^{-}}}\Big)^\beta\dt\bigg]\\
\nonumber&\leq\E\bigg[\int_0^T\Big((1+\|\nabla\bfv^N(t)\|_2^2)^{\frac{\lambda\beta}{1-\beta}}(1+\|\nabla\bfv^N(t)\|_{p^{-}})^{(2-p^{-})\frac{\beta}{1-\beta}}\dt\bigg]^{1-\beta}\\
\nonumber&\times\E\bigg[\int_0^T\frac{\|\nabla^2\bfv^N(t)\|_{p^{-}}^2}{(1+\|\nabla\bfv^N(t)\|_2^2)^{\lambda}(1+\|\nabla\bfv^N(t)\|_{p^{-}})^{2-p^{-}}}\dt\bigg]^\beta\\
\label{eq:frac1}&\leq \,C_1(\lambda_0,\Lambda_\bff)^\beta\E\big[I_1+I_2\big]^{1-\beta}
\end{align}
where 
\begin{align*}
I_1&=\int_0^T(1+\|\nabla\bfv^N(t)\|_{p^{-}})^{(2-p^{-})\frac{\beta}{1-\beta}}\dt,\\
I_2&=\int_0^T\|\nabla\bfv^N(t)\|_2^{\frac{2\lambda\beta}{1-\beta}}(1+\|\nabla\bfv^N(t)\|_{p^{-}})^{(2-p^{-})\frac{\beta}{1-\beta}}\dt.
\end{align*}
We can estimate $I_1$ by
\begin{align}\label{eq:frac2}
\begin{aligned}
\E[I_1]&\leq\,c\,\E\int_0^T\int_{\mt}\big(1+|\nabla\bfv^N(t)|^{p^{-}}\big)\dxt \\&\leq\,c\,\E\int_0^T\big(1+|\nabla\bfv^N(t)|^{p(\cdot)}\big)\dxt\leq\,C_1(\lambda_0,\Lambda_\bff)
\end{aligned}
\end{align}
using $(2-p^{-})\frac{\beta}{1-\beta}\leq p^-$ and Theorem \ref{thm:2.1'}.
For $I_2$ we use the interpolation
interpolation inequality
\begin{align*}
\|v\|_{2}\leq \|v\|_{p^-}^{\frac{(n+2)p^--2n}{2p^-}}\|v\|_{\frac{np^-}{n-p^-}}^{\frac{n(2-p^-)}{2p^-}},
\end{align*}
which holds for $p^-\in(\tfrac{2n}{n+2},2)$, and the continuous embedding 
\begin{align*}
W^{2,p^-}(\mt)\hookrightarrow W^{1,\frac{np^-}{n-p^-}}(\mt).
\end{align*}
As a consequence of Theorem \ref{thm:2.1'} (setting $\delta=\frac{2p^-}{n(2-p^-)}\frac{1-\beta}{\lambda}$) we can estimate $I_2$ by
\begin{align}
\nonumber
\E[I_2]&\leq\,c\,\E\int_0^T\|\nabla^2\bfv^N(t)\|_{p^-}^{\frac{n(2-p^-)}{p^-}\frac{\lambda\beta}{1-\beta}}(1+\|\nabla\bfv^N(t)\|_{p^-})^{\big[(2-p^-)+\frac{(n+2)p^--2n}{p^-}\lambda\big]\frac{\beta}{1-\beta}}\dt\\
\nonumber&\leq\,c\,\E\bigg(\int_0^T\|\nabla^2\bfv^N(t)\|_{p^-}^{2\beta}\bigg)^{\frac{1}{\delta}}\bigg(\int_0^T(1+\|\nabla\bfv^N(t)\|_{p^-})^{p^-}\dt\bigg)^{\frac{1}{\delta'}}\\
\nonumber&\leq\,\kappa\,\E\int_0^T\|\nabla^2\bfv^N(t)\|_{p^-}^{2\beta}\dt+c(\kappa)\,\E\int_0^T\int_{\mt}\big(1+|\nabla\bfv^N|^{p^-}\big)\dxt\\
\nonumber&\leq\,\kappa\,\E\int_0^T\|\nabla^2\bfv^N(t)\|_{p^-}^{2\beta}\dt+c(\kappa)\,\E\int_0^T\big(1+|\nabla\bfv^N|^{p(\cdot)}\big)\dxt\\
\label{eq:frac3}&\leq\,\kappa\,\E\int_0^T\|\nabla^2\bfv^N(t)\|_{p^-}^{2\beta}+C_1(\Lambda_0,\Lambda_\bff),
\end{align}
where $\kappa>0$ is arbitrary. Combining \eqref{eq:frac1}--\eqref{eq:frac3} and choosing $\kappa$ small enough we have shown
\begin{align}\label{eq:2311}
\begin{aligned}
\E\bigg[\int_0^T\|\nabla^2\bfv^N\|^{2\beta}_{p^-}\dt\bigg]
&\leq \,C_1(\Lambda_0,\Lambda_\bff).
\end{aligned}
\end{align}
In order to proceed we use the interpolation inequality
\begin{align*}
\|v\|_{1+\sigma,p^-}\leq \|v\|_{1,p^-}^{1-\sigma}\|v\|_{2,p^-}^{\sigma}
\end{align*}
for $\sigma=\frac{2\beta(p^--\overline p)}{\overline{p}(p^--2\beta)}$. We obtain
\begin{align*}
\E\int_0^T&\|\bfv^N\|_{1+\sigma,p^-}^{\overline{p}}\dt\leq \E\int_0^T\|\bfv^N\|_{1,p^-}^{(1-\sigma)\overline{p}}\|\bfv^N\|_{2,p^-}^{\sigma\overline{p}}\dt\\
&\leq \bigg(\E\int_0^T\|\bfv^N\|_{1,p^-}^{p^-}\dt\bigg)^{\frac{(1-\sigma)\overline{p}}{p^-}}\bigg(\E\int_0^T\|\bfv^N\|_{2,p^-}^{2\beta}\dt\bigg)^{1-\frac{(1-\sigma)\overline{p}}{p^-}}\leq\,C_1(\Lambda_0,\Lambda_\bff)
\end{align*}
as a consequence of Theorem \ref{thm:2.1'} and \eqref{eq:2311}.\\
If $p^-\geq2$ estimate \eqref{eq:2311} can be shown much easier. Indeed, we have by Theorems \ref{thm:2.1'} and \ref{thm:2.1'''}
\begin{align}\label{eq:2311}
\begin{aligned}
\E\bigg[&\int_0^T\|\nabla^2\bfv^N\|^{2\beta}_{2}\dt\bigg]=\E\bigg[\int_0^T(1+\|\nabla\bfv^N(t)\|_2^2)^{\lambda\beta}\frac{\|\nabla^2\bfv^N(t)\|_2^{2\beta}}{(1+\|\nabla\bfv^N(t)\|_2^2)^{\lambda\beta}}\dt\bigg]\\
&\leq \bigg[\E\int_0^T\frac{\|\nabla^2\bfv^N(t)\|_2^{2}}{(1+\|\nabla\bfv^N(t)\|_2^2)^{\lambda}}\dt\bigg]^\beta\bigg[\E\int_0^T(1+\|\nabla\bfv^N(t)\|_2^2)^{\frac{p^-}{2}}\dt\bigg]^{1-\beta}\\
&\leq \,C_1(\Lambda_0,\Lambda_\bff).
\end{aligned}
\end{align}
In order to proceed we use the interpolation inequality
\begin{align*}
\|v\|_{1+\sigma,\overline p}\leq \|v\|_{1,\overline p}^{1-\frac{\sigma}{s}}\|v\|_{1+s,\overline p}^{\frac{\sigma}{s}}
\end{align*}
which holds for any $0<\sigma<s$. Combining this with the embedding (recall that $\overline p<\frac{2n}{n-2}$)
\begin{align*}
W^{2,2}(\mt)\hookrightarrow W^{1+s,\overline p}(\mt),\quad s=\frac{2n-(n-2)\overline p}{2\overline p},
\end{align*}
we obtain for $\sigma=s\frac{2\beta(p^--\overline p)}{\overline p(p^--2\beta)}$
\begin{align*}
\E\int_0^T\|\bfv^N\|^{\overline p}_{1+\sigma,\overline p}\dt&\leq \E\int_0^T \|\bfv^N\|_{1,\overline p}^{\overline p(1-\frac{\sigma}{s})}\|\bfv^N\|_{2,2}^{\frac{\sigma\overline p}{s}}\dt\\
&\leq \bigg(\E\int_0^T\|\nabla\bfv^N(t)\|_{\overline p}^{p^-}\dt\bigg)^{\frac{\overline p}{p^-}(1-\frac{\sigma}{s})}\bigg(\E\int_0^T\|\nabla^2\bfv^N(t)\|_2^{2\beta}\dt\bigg)^{1-\frac{\overline p}{p^-}(1-\frac{\sigma}{s})}.
\end{align*}
The claim follows again from Theorem \ref{thm:2.1'} combined with Korn's inequality (recall that $\overline p<p^-$) and \eqref{eq:2311}.
\end{proof}

\subsection{Compactness}
\label{subsec:comp}
Before we can apply  the stochastic compactness method we need to gain some
information concerning the time regularity of $\bfv^N$. 
We go back to the system \eqref{eq:gal} and see that for any $\bfvarphi\in\hold_{\Div}^{\infty}(\tn)^{n}$ there holds
\begin{align}
\begin{split}
\int_{\tn} &\dd\bfv^N\cdot\mathcal P^N_\ell\bfvarphi\dx+\int_{\mt}\bfS(\cdot,\ep(\bfv^N)):\ep(\mathcal P^N_\ell\bfvarphi)\dx\dt\\&=\int_{\mt}\bfv^N\otimes\bfv^N:\nabla\mathcal P^N_\ell\bfvarphi\dx\dt\\&+\int_{\mt}\bff\cdot\mathcal P^N_\ell\bfvarphi\dx\dt+\int_{\mt} \Phi(\bfv^N)\,\dd W^N\cdot\mathcal P^N_\ell\bfvarphi\dx.
\end{split}
\end{align}
Here $\mathcal P^N_\ell$ denotes the orthogonal projection on $\mathcal X_N$ with respect to the $W^{\ell,2}(\mt)$ inner product, where $\ell$ is chosen such that $\sobo_{\Div}^{\ell,2}(\tn)\hookrightarrow
\sobo_{\Div}^{1,\infty}(\tn)$.
We now define for $t\in [0,T]$ the functionals $\mathcal{H}_{N}(t,\cdot)$ on $\hold_{\Div}^{\infty}(\tn)$ by 
\begin{align}\label{eq:defH}
\begin{split}
\mathcal{H}_{N}(t,\bfvarphi) & := -\int_{0}^{t}\int_{\tn}\mathbb{H}^{N}\colon \nabla\mathcal{P}^{N}\bfvarphi\dif x \dif\sigma,\qquad\bfvarphi\in\hold_{\Div}^{\infty}(\tn),
\end{split} 
\end{align}
where for $N\in\mathbb{N}$
\begin{align}\label{eq:defH1}
\mathbb{H}^{N}:= -\mathbf{S}_p(\cdot,\ep(\bfv^{N}))+\bfv^{N}\otimes\bfv^{N}-\nabla\Delta^{-1}\bff,
\end{align}
so that by Theorem~\ref{thm:2.1'} and the hypotheses collected in Definition~\ref{def:weak}, there holds 
\begin{align}\label{eq:811}
\mathbb{H}^{N}\in \lebe^{p_0}(\Omega,\mathscr{F},\mathbb{P};\lebe^{p_0}(0,T;\lebe^{p_0}(\tn)))
\end{align}
uniformly in $N\in\mathbb{N}$ for some $p_0>1$. Here, $\Delta^{-1}$ is the solution operator of the Poisson problem on the torus as has been recalled in section \ref{sec:functionspaces}. Now we claim that 
\begin{align}\label{eq:Hclaim1}
\sup_{N\in\mathbb{N}}\E\left[ \|\mathcal{H}^{N}\|_{\sobo^{1,p_{0}}([0,T];\sobo_{\Div}^{-\ell,p_{0}}(\tn))} \right]<\infty.
\end{align}
Recall that $\ell\in\mathbb{N}$ is chosen so large such that $\sobo_{\Div}^{\ell,2}(\tn)^{n}\hookrightarrow
\sobo_{\Div}^{1,\infty}(\tn)^{n}$. To see \eqref{eq:Hclaim1}, note that 
\begin{align*}
\lVert{\frac{\dif}{\dif t}\mathcal{H}^{N}(t,\cdot)\rVert}_{\lebe^{p_{0}}(0,T;\sobo_{\Div}^{-\ell,p_{0}}(\tn)))} & = \bigg\|\sup_{\|\bfvarphi\|_{\ell,p'_{0}}\leq 1}\frac{\dif}{\dif t}\mathcal{H}^{N}(t,\bfvarphi)\bigg\|_{\lebe^{p_{0}}(0,T)} \\
& = \bigg\|\sup_{ \|\bfvarphi\|_{\ell,p'_{0}}\leq 1}\int_{\tn}\mathbb{H}^{N}\colon\nabla\mathcal{P}^{N}_\ell\bfvarphi\dif x\bigg\|_{\lebe^{p_{0}}(0,T)} \\
& \leq \bigg\|\sup_{ \|\bfvarphi\|_{\ell,p'_{0}}\leq 1}\|\mathbb{H}^{N}(t,\cdot)\|_{\lebe^{p_{0}}}\|\nabla\mathcal{P}_\ell^{N}\bfvarphi\|_{\lebe^{p'_{0}}}\bigg\|_{\lebe^{p_{0}}(0,T)} \\
 & \leq C\left(\int_{0}^{T}\|\mathbb{H}^{N}(t,\cdot)\|_{\lebe^{p_{0}}}^{p_{0}}\dif\sigma\right)^{\frac{1}{p_{0}}}. 
\end{align*}
In consequence, raising the previous inequality to the $p_{0}$--th power and taking expectations in conjunction with \eqref{eq:811} gives \eqref{eq:Hclaim1}. On the other hand, we have for all $0\leq s<t\leq T$
\begin{align*}
\E\left[\lVert{\int_{0}^{t}\Phi(\bfv^{N})\dif W_{\sigma}^{N}  -\int_{0}^{s}\Phi(\bfv^{N})\dif W^{N}\rVert}_{\lebe^{2}(\tn)}^{\theta}\right]  & = \E\left[\bigg\|{\int_{s}^{t}\Phi(\bfv^{N})\dif W^{N}}_{\lebe^{2}(\tn)}^{\theta}\right]\\
& = \E\left[\lVert{\int_{s}^{t}\sum_{k=1}^{\infty}\Phi(\bfv^{N})e_{k}\dif\beta_{k}^{N}\rVert}_{\lebe^{2}(\tn)}^{\theta}\right] \\ & \leq \E\left[\lVert{\int_{s}^{t}\sum_{k=1}^{\infty}g_{k}(\bfv^{N})\dif\beta_{k}^{N}\rVert}_{\lebe^{2}(\tn)}^{2\cdot\frac{\theta}{2}}\right]\\
& =  \E\left[\Big(\int_{s}^{t}\sum_{k=1}^{\infty}\|g_{k}(\bfv^{N})\|_{\lebe^{2}(\tn)}^{2}\dif\sigma\Big)^{\frac{\theta}{2}}\right] \\ & \stackrel{\eqref{eq:phi}}{\leq} C\E\left[\Big(\int_{s}^{t}(1+\|\bfv^{N}\|_{\lebe^{2}(\tn)}^{2})\dif\sigma\Big)^{\frac{\theta}{2}}\right] \\
& =  C|t-s|^{\frac{\theta}{2}}\left(\E\left[\sup_{t\in (0,T)}(1+\|\bfv^{N}\|_{\lebe^{2}(\tn)}^{2}\right]\right)^\frac{\theta}{2} \\ &  \stackrel{\eqref{eq:regularityestimate1}}{\leq} C|t-s|^{\frac{\theta}{2}}. 
\end{align*}
At this point we are in position to apply the Kolmogorov continuity criterion to conclude that there exists $0<\kappa<1$ such that 
\begin{align}\label{eq:Hclaim2}
\sup_{N\in\mathbb{N}}\E\left[ \lVert{\int_{0}^{\cdot}\Phi(\bfv^{N})\dif W^{N}\rVert}_{\hold^{\kappa}([0,T];\lebe^{2}(\tn))} \right]<\infty. 
\end{align}
Let us note that since $\sobo_{\Div}^{\ell,p'_{0}}(\tn)\hookrightarrow \sobo_{\Div}^{1,2}(\tn)\hookrightarrow\lebe^{2}(\tn)$ and $1<p_{0}<\infty$ we have $\lebe^{2}(\tn)\hookrightarrow \sobo^{-\ell,p_{0}}(\tn)$. Hence \eqref{eq:Hclaim2} implies that $\E\left[ \lVert{\int_{0}^{\cdot}\Phi(\bfv^{N})\dif W^{N}\rVert}_{\hold^{\kappa}([0,T];\sobo^{-\ell,p_{0}}(\tn))}\right]$ is uniformly bounded in $N$. Combining this with \eqref{eq:Hclaim1}, a straightforward interpolation argument yields some $0<\mu<1$ such that 
\begin{align}\label{eq:Hclaim4}
\sup_{N\in\mathbb{N}}\E\left[\lVert{\bfv^{N}\rVert}_{C^\mu([0,T];\sobo_{\Div}^{-\ell,p_{0}}(\tn)^{n})}\right]<\infty. 
\end{align}
In view of compactness, let us now define the path space 
\begin{align}\label{eq:pathspace}
\begin{split}
\mathcal{X}&:=\mathcal X_{\bfv}\otimes \mathcal X_{p}\otimes \mathcal X_{\bff}\otimes \mathcal X_W,
\end{split}
\end{align}
where
\begin{align*}
\mathcal X_{\bfv}&:= C([0,T];\sobo_{\Div}^{-\ell,p_{0}}(\tn)^{n})\cap L^{\overline p}(0,T;W^{1,\overline p}_{\Div}(\mt)),\\
\mathcal X_{p}&:= \db{\hold^0([0,T]\times \mt)},\\
\mathcal X_{\bff}&:= L^2(0,T;W^{1,2}(\mt)),\\
\mathcal X_W&:=\hold([0,T];\mathfrak{U}_{0}).
\end{align*}
Here $\overline p$ is some fixed but arbitrary number in $\big(1,\min\{p^-,\frac{2n}{n-2}\}\big)$. 
We obtain the following.
\begin{proposition}\label{prop:bfutightness'}
The set $\{\mathcal{L}[\bfv^N,p,\bff,W];\,N\in\mathbb N\}$ is tight on $\mathcal{X}$.
\end{proposition}
\begin{proof}
By a fractional version of Aubin--Lions theorem (see \cite[Thm. 5.1.22]{FlGa})
we have compactness of the embedding
\begin{align}\label{eq:cpctAmann}
\begin{aligned}
 C^\mu([0,T];W_{\Div}^{-\ell,p_0}(\mt))&\cap \lebe^{\overline p}(0,T;\sobo_{\Div}^{1+\beta,\overline p}(\tn)^{n})\\ &\hookrightarrow\hookrightarrow \lebe^{\overline p}(0,T;\sobo_{\Div}^{1,\overline p}(\tn)^{n}).
\end{aligned}
\end{align}
On the other Arcel\`a-Ascoli's theorem yields
\begin{align*}
C^\mu([0,T];W^{-\ell,p_0}_{\Div}(\mt))\hookrightarrow \hookrightarrow C([0,T];W_{\Div}^{-\ell,p_0}(\mt)).
\end{align*}
So, we obtain tightness of $\mathcal L[\bfv^N]$ on $\mathcal X_\bfv$ from \eqref{eq:Hclaim4},
Corollary \ref{cor:frac} and Tschebyscheff's inequality.
\db{Tightness of the law of $p$ on $\mathcal X_p$ follows by using \eqref{eq:lawp2} and the compact embedding
$$\PPln(Q)\hookrightarrow\hookrightarrow \hold^0([0,T]\times\mt).$$
The latter one is a simple consequence of Arcela-Ascoli's theorem.}
Finally the laws of $\bff$ and $W$ of on their corresponding path spaces are tight as being Radon measures on Polish spaces.
\end{proof}

Prokhorov's
Theorem (see \cite[Thm.~2.6]{IkWa}) implies that $\{\mathcal{L}[\bfv^N,p,\bff,W];\,N\in\mathbb N\}$ is also relatively weakly compact. This means we have a weakly convergent subsequence. 
Now we use Skorohod's
representation theorem \cite[Thm.~2.7]{IkWa}
to infer the following result.

\begin{proposition}\label{prop:skorokhod}
There exists a complete probability space $(\tilde\Omega,\tilde\mf,\tilde\prst)$ with $\mathcal{X}$-valued Borel measurable random variables $(\tilde\bfv^N,\tilde p^N,\tilde \bff^N,\tilde W^N)$, $N\in\N$, and $(\tilde\bfv,\tilde p,\tilde \bff,\tilde W)$ such that (up to a subsequence)
\begin{enumerate}
 \item the law of $(\tilde\bfv^N,\tilde p^N,\tilde \bff^N,\tilde W^N)$ on $\mathcal{X}$ is given by $\mathcal{L}[\bfv^N, p,\bff,W]$, $N\in\N$,
\item the law of $(\tilde\bfv,\tilde p,\bff,\tilde W)$ on $\mathcal{X}$ is a Radon measure,
 \item $(\tilde\bfv^N,\tilde p^N,\tilde \bff^N,\tilde W^N)$ converges $\,\tilde{\prst}$-almost surely to $(\tilde\bfv,\tilde p,\tilde \bff,\tilde W)$ in the topology of $\mathcal{X}$, i.e.
\begin{align} \label{wWS116}
\begin{aligned}
\tilde \bfv^N &\to \tilde\bfv \ \mbox{in}\ C([0,T];W_{\Div}^{-\ell,p_0}(\mt)) \ \tilde\p\mbox{-a.s.}, \\
\tilde \bfv^N &\to \tilde\bfv \ \mbox{in}\ L^{\overline p}(0,T;W^{1,\overline p }_{\Div}(\mt)) \ \tilde\p\mbox{-a.s.}, \\
\tilde p^N &\to \tilde p \ \mbox{in}\ C^0([0,T]\times\mt) \ \tilde\p\mbox{-a.s.}, \\
\tilde \bff^N &\to \tilde \bff \ \mbox{in}\ L^2(0,T;W^{1,2}(\mt)) \ \tilde\p\mbox{-a.s.}, \\
\tilde W^N &\to \tilde W \ \mbox{in}\ C([0,T]; \mathfrak{U}_0 )\ \tilde\p\mbox{-a.s.}
\end{aligned}
\end{align}
\end{enumerate}
\end{proposition}

\subsection{Conclusion} 
\label{subsec:con}
The variables $\tilde\bfv,\tilde p,\tilde\bff,\tilde W$ are progressively measurable with respect to their canonical filtration, namely, 
$$\tilde\mf_t:=\sigma\bigg(\sigma[\tilde\bfv\db{|_{[0,t]}}]\cup\sigma[\tilde p\db{|_{[0,t]}}]\cup \sigma[\tilde\bff\db{|_{[0,t]}}]\cup\bigcup_{k=1}^\infty\sigma[\tilde  W_k\db{|_{[0,t]}}]\cup\db{\{N\in\tilde\mf:\tilde\p(N)=0\}}\bigg),\quad t\in[0,T].$$
In view of Lemma \cite[Chapter 2, Lemma 2.1.35]{BFHbook}, the process $\tilde W$ is a cylindrical Wiener processes with respect to its canonical filtration. It follows from Corollary \cite[Chapter 2, Corollary 2.1.36]{BFHbook} that $\tilde W$ is a cylindrical Wiener process with respect to $(\tilde\mf_t)_{t\geq 0}$.\\
Modifying slightly the proof, the result of \cite[Chapter 2, Theorem 2.9.1]{BFHbook} remains valid in the current setting. Hence, as a consequence of the equality of laws from Proposition~\ref{prop:skorokhod}, the approximate equation \eqref{eq:gal} is satisfied on the new probability space, i.e. we have
\begin{align*}
\int_{\mt} &\tilde\bfv^N\cdot\bfw_k\dx+\int_0^t\int_{\mt}\mu(1+|\ep(\tilde\bfv^N)|)^{\tilde p^N(\cdot)-2}\ep(\tilde\bfv^N):\ep(\bfw_k)\dx\dt\\&=\int_{\mt} \tilde\bfv^N(0)\cdot\bfw_k\dx+\int_0^t\int_{\mt}\tilde\bfv^N\otimes\tilde\bfv^N:\nabla\bfw_k\dx\dt\\&+\int_0^t\int_{\mt}\tilde\bff^N\cdot\bfw_k\dx\dt+\int_0^t\int_{\mt} \Phi(\tilde\bfv^N)\,\dd \tilde W^N\cdot\bfw_k\dx
\end{align*}
$\tilde\p$-a.s. for all $t\in[0,T]$. Using the convergence from \eqref{wWS116} it is easy to pass to the limit and we obtain
\begin{align}\label{eq:LIMIT}
\begin{aligned}
\int_{\mt}\tilde\bfv(t)\cdot\bfvarphi\dx &+\int_0^t\int_{\mt}\mu(1+|\ep(\tilde \bfv)|)^{\tilde p(\cdot)-2}\ep(\tilde\bfv):\ep(\bfphi)\dxs\\&=\int_{\mt}\tilde\bfv(0)\cdot\bfvarphi\dx+\int_0^t\int_{\mt}\tilde \bfv\otimes\tilde\bfv:\ep(\bfphi)\dxs\\
&+\int_{\mt}\int_0^t\tilde \bff\cdot\bfphi\dxs+\int_{\mt}\int_0^t\Phi(\tilde\bfv)\,\dd \tilde W\cdot \bfvarphi\dx
\end{aligned}
\end{align}
for all $\bfvarphi\in C^\infty_{\Div}(\mt)$ and all $t\in[0,T]$ $\tilde\p$-a.s. where, for the limit passage in the stochastic integral, we use \cite[Lem.~2.1]{debussche1}.

\subsection{\db{Stochastically strong solutions}}
\label{subsec:path}
Let us start by showing pathwise uniqueness. 
\begin{proposition}[Pathwise uniqueness]
\label{prop:unique}
Let the assumptions of Theorem \ref{thm:main3} be valid. In particular, we suppose $p^-\geq \frac{n+2}{2}$.
Let $\bfv^1$, $\bfv^2$ be two \db{weak stochastically strong solutions} to \eqref{0.4}--\eqref{0.3} in the sense of Definition \ref{def:strsol} defined on the same stochastic basis with the same Wiener process $W$, the same forcing $\bff$ and the same exponent $p$. If
$$\mathbb{P} \left[  \bfv^1(0) = \bfv^2(0) \right] = 1,$$
then
\[
\mathbb{P} \left[  \bfv^1(t) = \bfv^2(t),  \ \mbox{for all}\ t \in [0,T]\right] = 1.
\]
\end{proposition}
\begin{proof}
We set $\bfw=\bfv^1-\bfv^2$ and apply It\^{o}'s formula
to $\bfw\mapsto \frac{1}{2}\int_{\mt}|\bfw|^2\dx$ (recall that by our assumptions on $p^-$ the term $\int_{\mt}\bfv\otimes\bfv:\nabla\bfv\dx$ is well-defined).
\db{This procedure can be made rigorous by applying a regularization to the equation for $\bfw$. Eventually, the standard one-dimensional It\^o formula can be applied to $|\bfw_\varrho|^2$ pointwise in $x$, where $\varrho$ is the regularization parameter.
Smooth approximations converge in $L^p(Q)$ and $L^{p'}(Q)$ as we have $p\in\PPln(Q)$ $\p$-a.s. by assumption, cf. \cite[Thm. 9.1.7]{DiHaHaRu}}. We obtain
using $\bfw(0)=0$
\begin{align*}
\frac{1}{2}\|\bfw(t)\|_{\lebe^2(\mt)}^2
&=-\int_0^t\int_{\mt}\Big(\bfS_p(\cdot,\ep(\bfv^1))-\bfS_p(\cdot,\ep(\bfv^2))\Big):\ep(\bfv^1-\bfv^2)\dxs\\
&+\int_0^t\int \big((\nabla \bfv^1)\bfv^1-(\nabla \bfv^2)\bfv^2)\cdot\bfw\dx\\
&+\frac{1}{2}\int_{\mt}\int_0^t\,\dd\Big\langle\Big\langle\int_0^\cdot\big(\Phi(\bfv^1)-\Phi(\bfv^2)\big)\,\dd W\Big\rangle\Big\rangle\dx\\
&+\int_{\mt}\int_0^t\bfw\cdot \big(\Phi(\bfv^1)-\Phi(\bfv^2)\big)\,\dd W\dx.
\end{align*}
By monotonicity of $\bfS_p$ the first term on the right-hand side is non-negative and we have by Korn's inequality
\begin{align*}
\int_{\mt}\Big(\bfS_p(\cdot,\ep(\bfv^1))-\bfS_p(\cdot,\ep(\bfv^2))\Big):\ep(\bfv^1-\bfv^2)\dx\geq\,\mu\|\ep(\bfw)\|_2^2\geq\,\tfrac{\mu}{c}\|\nabla\bfw\|_2^2
\end{align*}
as $p^-\geq2$. The critical part is the term arising from the convective term. Here, we follow ideas of \cite{MNRR}[Thm. 4.29] and write
\begin{align*}
\int_{\mt} \big((\nabla \bfv^1)\bfv^1-(\nabla \bfv^2)\bfv^2)\cdot(\bfv^1-\bfv^2)\dx=\int_{\mt} (\nabla \bfv^1)\bfw\cdot \bfw\dx\leq\|\nabla\bfv^1\|_{p^-}\|\bfw\|^2_{\frac{2p^-}{p^--1}}.
\end{align*}
Now, we use the interpolation
$$\|v\|_q\leq\|v\|_2^\alpha\|\nabla v\|_2^{1-\alpha},\quad \alpha=\frac{2n-q(n-2)}{2q},$$
valid for all $q\in[2,\frac{2n}{n-2}]$ if $n\geq3$ and $q\in[2,\infty)$ if $n=2$, cp. \cite{MNRR}[Lemma 4.35]. Choosing $q=\frac{2p^-}{p^--1}$ we obtain
\begin{align*}
\int_{\mt} \big((\nabla \bfv^1)\bfv^1-(\nabla \bfv^2)\bfv^2)\cdot(\bfv^1-\bfv^2)\dx&\leq \|\nabla\bfv^1\|_{p^-}\|\bfw\|^{\frac{2p^--n}{p^-}}_2\|\nabla\bfw\|^{\frac{n}{p^-}}_2\\
&\leq\mu\|\nabla \bfw\|_2^2+c(\mu)\|\nabla\bfv^1\|_{p^-}^{\frac{2p^-}{2p^--n}}\|\bfw\|^2_2
\end{align*}
using also Young's inequality.
Finally, we estimate the correction term by
\begin{align*}
\int_{\mt}\int_0^t\,\dd\Big\langle\Big\langle\int_0^\cdot\big(\Phi(\bfv^1)-\Phi(\bfv^2)\big)\,\dd W\Big\rangle\Big\rangle\dx&=\sum_{k=1}^\infty\int_0^t\bigg(\int_{\mt}\big(g_k(\bfv^1)-g_k(\bfv^2)\big)\dx\bigg)^2\ds\\
&\leq\sum_{k=1}^\infty\int_0^t\int_{\mt}\big|g_k(\bfv^1)-g_k(\bfv^2)\big|^2\dx\ds\\
&\leq\,\int_0^t\int_{\mt}|\bfv^1-\bfv^2|^2\dxs
\end{align*}
using \eqref{eq:phi}.
Summarising, we obtain 
\begin{align}\label{eq:uniquenessbound}
\dif\|\bfw\|_{\lebe^{2}}^{2}\leq \,c\Big(\|\nabla \bfv^1\|_{p^-}^{\frac{2p^-}{2p^--n}}+1\Big)\|\bfw\|_{2}^{2}\dt+\int_{\mt}\bfw\cdot \big(\Phi(\bfv^1)-\Phi(\bfv^2)\big)\dif W\dx
\end{align}
 for some finite constant $c>0$. We now define $G\colon\Omega\times[0,T]\to\R$ by $$G(\omega,t):=c\Big(\|\nabla \bfv^1(\omega,t)\|_{p^-}^{\frac{2p^-}{2p^--n}}+1\Big)$$ so that in particular $G\in\lebe^{1}(0,T)$ for $\p$--a.e. $\omega\in\Omega$. This is a consequence of $\frac{2p^-}{2p^--n}\leq p^-$ (which follows from the assumption $p^-\geq\frac{n+2}{n}$) and $\nabla\bfv^1\in L^{p^-}(Q)$ $\p$-a.s. (which follows from $\ep(\bfv^1)\in L^{p(\cdot)}(Q)$ $\p$-a.s. and Korn's inequality). We then obtain by use of It\^{o}'s formula (similar to \cite{schm})
\begin{align}
\begin{split}
\dif\left(\e^{-\int_{0}^{t}G\dif s}\|\bfw\|_{L^{2}}^{2}\right) & \,\,\,= -G\e^{-\int_{0}^{t}G\dif s}\|\bfw\|_{L^{2}}^{2}\dif t + \e^{-\int_{0}^{t}G\dif s}\dif\|\bfw\|_{L^{2}}^{2}\\
&\stackrel{\eqref{eq:uniquenessbound}}{\leq}  \e^{-\int_{0}^{t}G\dif s}\int_{\mt}\bfw\cdot \big(\Phi(\bfv^1)-\Phi(\bfv^2)\big)\dif W\dx
\end{split}
\end{align}
by definition of $G$. Now we apply the expectation to both sides of the inequality and consequently obtain
\begin{align*}
\E\left[\e^{-\int_{0}^{t}G\dif s}\|\bfw\|_{L^{2}}^{2}\right]=0. 
\end{align*}
Consequently we obtain $\bfv^1=\bfv^2$ $\p$-a.s. and the proof of Proposition \ref{prop:unique} is complete.
\end{proof}
Based on the pathwise uniqueness we will employ 
the Gy\"{o}ngy--Krylov characterization of convergence in probability introduced in \cite{krylov}. It applies to situations when pathwise uniqueness and existence of a martingale solution are valid and allows to establish existence of a \db{stochastically strong solution}. We consider two sequences $(N_n),(N_m)\subset \N$ diverging to infinity. Let
$\bfv^n:=\bfv^{N_n}$ and $\bfv^m:=\bfv^{N_m}$.
We consider the collection of joint laws of $(\bfv^n,\bfv^m,p,\bff,W)$
on the extended path space
$$\mathcal{X}^J=\mathcal X_{\bfv}^2\otimes \mathcal X_{p}\otimes \mathcal X_{\bff}\db{\otimes}\mathcal{X}_W,$$ 
Similarly to Proposition \ref{prop:bfutightness'} we obtain the following result.

\begin{proposition}\label{prop:tight}
The set
$$\{\mathcal L[\bfv^n,\bfv^m,p,\bff,W];\,n,m\in\mn\}$$
is tight on $\mathcal{X}^J$.
\end{proposition}

Let us take any subsequence $(\bfv^{n_k},\bfv^{m_k},p,\bff,W)$. By the Skorokhod representation theorem we infer (for a further subsequence but without loss of generality we keep the same notation) the existence of a probability space $(\bar{\Omega},\bar{\mf},\bar{\prst})$ with a sequence of random variables $(\hat\bfv^{n_k},\check\bfv^{m_k},\bar{p}_k,\bar{\bff}_k,\bar{W}_k)$
conver\-ging almost surely in $\mathcal{X}^J$ to a random variable $(\hat\bfv,\check\bfv,\bar p,\bar\bff,\bar{W})$.
Moreover,
$$\mathcal L[\hat\bfv^{n_k},\check\bfv^{m_k},\bar p^k,\bar\bff^k,\bar{W}^k]=\mathcal L[\bfv_{n_k},\bfv_{m_k},p,\bff,W]$$
on $\mathcal{X}^J$ for all $k\in\mathbb N$.
%
Observe that in particular, $\mathcal L[\bfv_{n_k},\bfv_{m_k},\bar p^k,\bar\bff^k,\bar{W}^k]$ converges weakly to the measure $\mathcal L[\hat\bfv,\check\bfv,\bar p,\bar\bff,\bar{W}]$.
As in \eqref{eq:LIMIT} we can
show that $(\hat\bfv,\bar p,\bar \bff,\bar W)$ and $(\check\bfv,\bar p,\bar \bff,\bar W)$ are weak martingale solutions to \eqref{0.4}--\eqref{0.3} defined on the same stochastic basis $(\bar{\Omega},\bar{\mf},(\bar{\mf}_t),\bar{\prst})$, where $(\bar{\mf}_t)_{t\geq0}$ is the $\bar\prst$-augmented canonical filtration of $(\hat\bfv,\check\bfv,\bar p,\bar \bff,\bar W)$.
We employ the pathwise uniqueness result from Proposition \ref{prop:unique}. Indeed, it follows from our assumptions on the approximate initial laws $\Lambda_0$ that $\hat{\bfv}(0)=\check{\bfv}(0)=1$ $\bar\p$-a.s. Therefore, the solutions $\hat\bfv$ and $\check\bfv$ coincide $\bar\p$-a.s. and we have
\begin{align*}
\mathcal L&[\hat\bfv,\check\bfv,\bar{W}]\Big((\bfv_1,\bfv_2,p,\bff,W)\in\mathcal X^J:\;\bfv_1=\bfv_2\Big)
=\bar{\prst}(\hat{\bfv}=\check{\bfv})=1.
\end{align*}
Now, we have all in hand to apply the Gy\"ongy--Krylov theorem.
It implies that the original sequence $\bfv^N$ defined on the initial probability space $(\Omega,\mf,\prst)$ converges in probability in the topology of $\mathcal{X}_{\bfv}$ to the random variable $\bfv$. Therefore, we finally deduce that $\bfv$ is a weak \db{stochastically strong} solution to \eqref{0.4}--\eqref{0.3}.
\hfill $\Box$

\section{Analytically Strong Solutions}
\label{sec:2d}

\subsection{A--Priori Bounds}
In this section we establish the existence result, Theorem \ref{thm:main}. We begin with a strengthening of the a--priori estimate given by Theorem~\ref{thm:2.1'}. Note that we work under the additional assumption that either we have
\begin{itemize}
\item[(i)] $n=2$ and $1<p^-\leq  p^+< 4$ or;
\item[(ii)] $n=3$ and $\frac{11}{5}<p^-\leq p^+\leq p^-+\frac{4}{5}$.
\end{itemize}
\begin{theorem}\label{thm:2.1''}
Let the assumptions of Theorem \ref{thm:main2} be satisfied. Let $\bfv^N$ be the Galerkin approximation constructed in Section \ref{sec:galerkin}.
Then there exists a constant $c>0$ such that
\begin{align}\label{eq:2Duniformest}
\begin{split}
\E&\bigg[\sup_{t\in(0,T)}\int_{\mt} |\nabla\bfv^N(t)|^2\dx+\int_{\Q} |\nabla_{\bfxi}\bfF_p(\cdot,\ep(\bfv^N))|^{2}\dxt\bigg]\\&\leq c\,\Big(
\int_{\lebe^2_{\Div}(\mt)}\big\|\bfu\big\|_{\lebe^2(\mt)}^{2}\,\dd\Lambda_0(\bfu),
\int_{\lebe^2(\Q)}\big\|\bfg\big\|_{\lebe^2(0,T;\sobo^{1,2}(\mt))}^{2}\,\dd\Lambda_\bff(\bfg)\Big)
\end{split}
\end{align}
uniformly in $N\in\mathbb{N}$. 
\end{theorem}
\begin{corollary}\label{cor:2.1''}
Under the assumptions of Theorem \ref{thm:2.1''} we have
\begin{align}\label{eq:2Duniformestcor}
\begin{split}
\E&\bigg[\int_{\Q} |\nabla^2\bfv^N|^{\min(p^-,2)}\dxt\bigg]\\&\leq c\,\Big(
\int_{\lebe^2_{\Div}(\mt)}\big\|\bfu\big\|_{\lebe^2(\mt)}^{2r}\,\dd\Lambda_0(\bfu),
\int_{\lebe^2(\Q)}\big\|\bfg\big\|_{\lebe^2(0,T;\sobo^{1,2}(\mt))}^{2r}\,\dd\Lambda_\bff(\bfg)\Big)
\end{split}
\end{align}
uniformly in $N\in\mathbb{N}$. 
\end{corollary}
\begin{proof}[Proof of Corollary~\ref{cor:2.1''}]
If $p^-\geq 2$ the claim follows immediately from Theorem~\ref{thm:2.1''}, the definition of $\bfF_p$ and Korn's inequality. So, let us assume that $p^-<2$.
By Korn's and Young's inequality we obtain
\begin{align*}
\E\bigg[&\int_{\Q} |\nabla^2\bfv^N|^{p^-}\dxt\bigg]\leq\,c\,\E\bigg[\int_{\Q} |\nabla\ep(\bfv^N)|^{p^-}\dxt\bigg]\leq\,c\,\E\bigg[1+\int_{\Q} |\nabla\ep(\bfv^N)|^{p}\dxt\bigg] \\
&=\,c\,\E\bigg[1+\int_{\Q}(1+|\ep(\bfv^N)|)^{p\frac{2-p}{2}}(1+|\ep(\bfv^N)|)^{p\frac{p-2}{2}} |\nabla\ep(\bfv^N)|^{p}\dxt\bigg]\\
&\leq\,c\,\E\bigg[\int_{\Q}(1+|\ep(\bfv^N)|)^{p}\dxt+\int_{\Q}(1+|\ep(\bfv^N)|)^{p-2} |\nabla\ep(\bfv^N)|^{2}\dxt\bigg].
\end{align*}
Now, the first term is bounded by Theorem~\ref{thm:2.1'} and the second one by
Theorem~\ref{thm:2.1''}. Clearly, $c>0$ does not depend on $N$, and hence the statement of the corollary follows. 
\end{proof}
\begin{proof}[Proof of Theorem \ref{thm:2.1''}]
In a similar vein as for Theorem~\ref{thm:2.1'}, the core of the proof consists in a suitable It\^{o}--expansion. We  hereafter apply It\^{o}'s formula to the function $f_{\gamma}(\bfu):=\tfrac{1}{2}\|\partial_{\gamma}\bfu\|_{\lebe^2(\mt)^{n}}^2$ (with $\gamma\in\{1,2\}$ for $n=2$ and $\gamma\in\{1,2,3\}$ for $n=3$) and obtain
\begin{align}\nonumber
\frac{1}{2}\|\partial_{\gamma}\bfv^N(t)\|_{\lebe^2(\mt)}^2&=\frac{1}{2}\|\partial_{\gamma}\mathcal P^N\bfv_0\|_{\lebe^2(\mt)}^2+\int_0^t f'(\bfv^N)\,\dd\bfv^N_\sigma+\frac{1}{2}\int_0^t f''(\bfv^N)\,\dd\langle\bfv^N\rangle_\sigma\\
\label{eq:2Duniformest1}&=\frac{1}{2}\|\partial_{\gamma}\mathcal P^N\bfv_0\|_{\lebe^2(\mt)}^2+\int_{\mt}\int_0^t \partial_{\gamma}\bfv^N\cdot\dd\partial_{\gamma}\bfv^N_\sigma\dx\\&+\frac{1}{2}\int_{\mt}\int_0^t\dd\Big\langle\!\Big\langle\int_0^{\cdot}\partial_{\gamma} \big(\Phi(\bfv^N)\,\dd\bfW\big)\Big\rangle\!\Big\rangle_\sigma\dx
=:(I)+(II)+(III).
\nonumber
\end{align}
We consider the three integrals separately. \\
\noindent
\emph{1.} We begin with $(I)$. By continuity of the projection, we record the estimate
\begin{align*}
\|\partial_{\gamma}\mathcal P^N\bfv_0\|_{\lebe^2(\mt)}^2 & \leq \|\mathcal P^N\bfv_0\|_{\sobo^{1,2}(\mt)}^2 \leq C \| \bfv_0\|_{\sobo^{1,2}(\mt)}^2. 
\end{align*}
\emph{2.} Deferring the estimation of $(III)$ to the end of the proof, we turn to $(II)$. Summing over $\gamma$, we find  
\begin{align*}
(II)&=-(II)_1-(II)_2+(II)_3+(II)_{4}+(II)_{5},\\
(II)_1&:=\int_0^t\int_{\mt} D_\bfxi\bfS(\cdot,\ep(\bfv^N))(\partial_{\gamma}\ep(\bfv^N),\partial_{\gamma}\ep(\bfv^N))\dxs,\\
(II)_2&:=\int_0^t\int_{\mt} \partial_\gamma\bfS(\cdot,\ep(\bfv^N)):\partial_{\gamma}\ep(\bfv^N)\dxs,\\
(II)_3&:=\int_0^t\int_{\mt}\partial_{\gamma}\bfv^N
\cdot\partial_{\gamma}\Big(\Phi(\bfv^N)\,\dd W_\sigma\Big)\dx,\\
(II)_{4}&:=\int_{0}^{t}\int_{\mt}\partial_{\gamma}\vv^{N}\cdot\partial_{\gamma}\ff\dif x\dif\sigma,\\
(II)_{5}&:=\int_{\mt}\Div(\bfv^N\otimes\bfv^N)\cdot\partial_{\gamma}^{2}\bfv^N\dx.
\end{align*}
\emph{Ad $(II)_{1}$.} Using the assumptions for $\bfS$ in \eqref{0.3} we obtain
\begin{align}\label{eq:II1estimate}
\begin{split}
(II)_1
&\geq \widetilde{c}\int_0^t\int_{\mt}(1+|\ep(\bfv^N)|^{2})^{\frac{p(\cdot)-2}{2}}
|\partial_{\gamma}\ep(\bfv^N)|^2\dxs.
\end{split}
\end{align}
\emph{Ad $(II)_{2}$.} We now turn to the second term $(II)_{2}$. By uniform Lipschitz continuity of $p(\omega,\cdot)$ we obtain
\begin{align}\label{eq:coercivepotential}
|\partial_{\gamma}\mathbf{S}(\cdot,\ep(\bfv^{N}))| \leq c\ln(1+|\ep(\bfv^{N})|)(1+|\ep(\vv^{N})|)^{p(\cdot)-2}|\ep(\bfv^{N})|
\end{align}
with an absolute constant $c>0$ for all $N\in\mathbb{N}$. We find by virtue of Young's Inequality for arbitrary $\delta>0$
\begin{align*}
(II)_2&\leq \,c\,\bigg(1+ \int_0^t\int_{\mt}\ln(1+|\ep(\bfv^N)|)(1+|\ep(\bfv^N)|^{p(\cdot)-1})|\partial_{\gamma}\ep(\bfv^N)|\dxs\bigg)\\
&\leq \,c(\delta)\,\bigg(1+ \int_0^t\int_{\mt}\ln^{2}(1+|\ep(\bfv^N)|)(1+|\ep(\bfv^N)|^{p(\cdot)})\dxs\bigg)\\&+ \delta\,\bigg(\int_0^t\int_{\mt}(1+|\ep(\bfv^N)|^{p(\cdot)-2})|\partial_{\gamma}\ep(\bfv^N)|^2\dxs\bigg) = c(\delta)\mathbf{I}'+\delta\mathbf{II}'.
\end{align*}
Choosing $\delta>0$ sufficiently small, $\delta\mathbf{II}'$ may be absorbed into the left side of the overall inequality by the coercive estimation of $(II)_{1}$ (cp.~\eqref{eq:II1estimate}), and therefore it remains to give a suitable upper bound for $c(\delta)\mathbf{I}'$. It is easy to see that for every $2<\mu<3$ there exists a constant $C=C(\mu)>0$ such that for all $t>0$ there holds $t^{2}\log^{2}(1+t)\leq C(1+t^{\mu})$. Using the Gagliardo--Nirenberg interpolation inequality on the torus \cite[Thm 7.28]{GT}, we obtain for $1\leq q,r\leq \infty$, $0\leq\alpha\leq 1$ the implication 
\begin{align}\label{eq:GNS}
\frac{1}{p}=\Big(\frac{1}{r}-\frac{1}{n}\Big)\alpha+\frac{1-\alpha}{q} \Longrightarrow \|u\|_{\lebe^{p}}\leq C\|u\|_{\sobo^{1,r}}^{\alpha}\|u\|_{\lebe^{q}}^{1-\alpha}\;\;\text{for}\;u\in(\sobo^{1,r}\cap\lebe^{q})(\tn), 
\end{align}
where $C>0$ only depends on $q,r$ and $n$. 
%
Now set $p=\mu$, $q=2$ and $r=2$, so that the condition in \eqref{eq:GNS} is satisfied with $\alpha=\frac{\mu-2}{\mu}$. Then we have $1-\alpha=\frac{2}{\mu}$ and so by Young's inequality with $\delta>0$ to be fixed later
\begin{align}
\begin{split}
\|v\|_{\lebe^{\mu}(\tn)}^{\mu}
&  \leq C\|v\|_{\sobo^{1,2}}^{\mu-2}\|v\|_{\lebe^{2}}^{2} \\ & \leq C\Big(\delta\|v\|_{\sobo^{1,2}}^{2}+C_{\delta}\|v\|_{\lebe^{2}}^{\frac{4}{4-\mu}}\Big)\\
& = C\Big(\delta\big(\|v\|_{\lebe^{2}}^{2}+\|\nabla v\|_{\lebe^{2}}^{2}\big)\Big)+C_{\delta}\|v\|_{\lebe^{2}}^{\frac{4}{4-\mu}}\Big)\qquad\text{for every}\;v\in\sobo^{1,2}(\tn). 
\end{split}
\end{align}
This estimation is implicit in \cite[cp. Eq. (4.62)]{Di}. We apply the previous estimate to $v:=(1+|\ep(\vv^{N})|^{2})^{p(\cdot)/4}$ to find
\begin{align}
\begin{split}
\mathbf{I}' & \leq C\int_{0}^{t}\int_{\mt} (1+|\ep(\vv^{N})|^{2})^{\mu p(\cdot)/4}\dxs \\
& \leq C + C\delta\int_{0}^{t}\Big(\|(1+|\ep(\vv^{N})|^{2})^{p(\cdot)/4}\|_{\lebe^{2}}^{2}+\|\nabla (1+|\ep(\vv^{N})|^{2})^{p(\cdot)/4}\|_{\lebe^{2}}^{2}\Big)\dif\sigma\\ & +C_{\delta}\int_{0}^{t}\|(1+|\ep(\vv^{N})|^{2})^{p(\cdot)/4}\|_{\lebe^{2}}^{\frac{4}{4-\mu}}\Big)\dif\sigma\\
& \leq  C + C\delta\int_{Q_{T}}|\ep(\vv^{N})|^{p(\cdot)}\dxs +C\delta\int_{0}^{t}\|\nabla (1+|\ep(\vv^{N})|^{2})^{p(\cdot)/4}\|_{\lebe^{2}}^{2}\dif\sigma\\ & +C_{\delta}\int_{0}^{t}\|(1+|\ep(\vv^{N})|^{2})^{p(\cdot)/4}\|_{\lebe^{2}}^{\frac{4}{4-\mu}}\Big)\dif\sigma.
\end{split}
\end{align}
By \eqref{eq:coercivepotential}, we obtain
\begin{align*}
\int_{0}^{t}\|\nabla (1+|\ep(\vv^{N})|^{2})^{p(\cdot)/4}\|_{\lebe^{2}}^{2}\dif\sigma & \leq C \int_{0}^{t}\int_{\tn}(1+|\ep(\vv^{N})|^{2})^{\frac{p(\cdot)}{2}}\ln^{2}\big(1+|\ep(\vv^{N})|^{2}\big)\dxs\\
& + C\int_0^t\int_{\mt}(1+|\ep(\bfv^N)|^{2})^{\frac{p(\cdot)-2}{2}}
|\partial_{\gamma}\ep(\bfv^N)|^2\dxs.
\end{align*}
So that, choosing $\delta>0$ small enough and absorbing the first term of the right side of the previous inequality into $\mathbf{I}'$, we end up with (recall $2<\mu<3$ so that $4/(4-\mu)\leq 4$), 
\begin{align*}
\mathbf{I}' & \leq  C  + C\delta\int_{Q_{T}}|\ep(\vv^{N})|^{p(\cdot)}\dxs +C\delta\int_0^t\int_{\mt}(1+|\ep(\bfv^N)|^{2})^{\frac{p(\cdot)-2}{2}}
|\partial_{\gamma}\ep(\bfv^N)|^2\dxs\\ & +C_{\delta}\int_{0}^{t}\|(1+|\ep(\vv^{N})|^{2})^{p(\cdot)/4}\|_{\lebe^{2}}^{4}\dif\sigma\\
&\leq C+ C\,\int_{Q_{T}}|\ep(\vv^{N})|^{p(\cdot)}\dxs +C\delta\int_0^t\int_{\mt}(1+|\ep(\bfv^N)|^{2})^{\frac{p(\cdot)-2}{2}}
|\partial_{\gamma}\ep(\bfv^N)|^2\dxs\\ & +C_{\delta}\int_{0}^{t}\int_{\mt}|\ep(\vv^{N})|^{p(\cdot)\frac{2}{4-\mu}}\dx\dif\sigma = \mathbf{I}'_{1}+...+\mathbf{I}'_{4}. 
\end{align*}
The terms $\mathbf{I}'_{1}$ and $\mathbf{I}'_{2}$ are already in a convenient form. For $\delta$ small enough consequently may absorb $\mathbf{I}'_{3}$ into the right side of \eqref{eq:II1estimate}. 

\emph{Ad $(II)_{3}$.} We decompose
\begin{align*}
(II)_3&=\int_{\mt}\int_0^t\partial_\gamma\bfv^N\cdot\partial_\gamma\Big(\Phi(\bfv^N)e_k\,\dd\beta_k\Big)\dx\\
&=\sum_k\int_{\mt}\int_0^t\partial_\gamma\bfv^N\cdot\partial_\gamma\Big(g_k(\bfv^N)\,\dd\beta_k\Big)\dx\\
&=\sum_k\int_{\mt}\int_0^t \nabla_{\bfxi} g_k(\bfv^N)
(\partial_\gamma\bfv^N,\partial_\gamma\bfv^N)\,\dd\beta_k\dx\\
&+\sum_k\int_{\mt}\int_0^t\partial_\gamma\bfv^N\cdot \partial_\gamma g_k(\bfv^N)\,\dd\beta_k\dx\\
&=\int_{\mt}\int_0^t\mathcal G^\bfxi
(\partial_\gamma\bfv^N,\partial_\gamma\bfv^N)\,\dd\beta_k\dx.
\end{align*}
On account of assumption (\ref{eq:phi}), Burkholder-Davis-Gundy inequality and Young's inequality we obtain for arbitrary $\delta>0$
\begin{align*}
\E\bigg[\sup_{0<t<T}|(II)_3^1|\bigg]&\leq \E\bigg[\sup_{0<t<T}\Big|\int_0^t\sum_k\int_{\mt}\nabla g_k(\bfv_N)
(\partial_\gamma\bfv^N,\partial_\gamma\bfv^N)\dx\,\dd\beta_k\Big|\bigg]\\
&\leq c\,\E\bigg[\int_0^T\bigg(\int_{\mt}  \nabla g_k(\bfv^N)
(\partial_\gamma\bfv^N,\partial_\gamma\bfv^N)\dx\bigg)^2\dt\bigg]^{\frac 12}\\
&\leq c\,\E\bigg[\bigg(\int_0^T\bigg(\int_{\mt}
|\partial_\gamma\bfv^N|^2\dx\bigg)^2\dt\bigg]^{\frac 12}\\
&\leq \delta\,\E\bigg[\sup_{0<t<T}\int_{\mt}
|\partial_\gamma\bfv^N|^2\dx\bigg]+ c(\delta)\,\E\bigg[\int_\Q
|\partial_\gamma\bfv^N|^2\dxt\bigg].
\end{align*}
\emph{Ad $(II)_{4}$.} After we shall have passed to the supremum in the overall inequality, by Young's inequality we obtain for a finite constant $C_{\delta}>0$
\begin{align*}
(II)_{4}\leq C\delta \sup_{0<t<T}\int_{\tn}|\partial_{\gamma}\vv^{N}|^{2}\dif x + C_{\delta}\int_{0}^{t}\int_{\tn}|\nabla \ff|^{2}\dxs.
\end{align*}
We then may choose $\delta>0$ so small such that $\delta\|\partial_{\gamma}\vv^{N}\|_{\lebe^{2}(Q)}^{2}$ can be absorbed into \eqref{eq:2Duniformest1}.
\noindent
\emph{Ad $(III)$.} We have by (\ref{eq:phi})
\begin{align*}
(III)&=\frac{1}{2}\int_{\mt}\int_0^t \dd\Big\langle\Big\langle\int_0^{\cdot}\partial_{\gamma} \big(\Phi(\bfv^N)\,\dd W^{N}\big)\Big\rangle\Big\rangle_\sigma\dx\\
&\leq \frac{1}{2}\sum_{k}\int_{\mt}\int_0^t \,\dd\Big\langle\Big\langle\int_0^{\cdot}\partial_{\gamma}\big(\Phi(\bfv^N)\bfe_k\big) \,\dd\beta_k\Big\rangle\Big\rangle_\sigma\dx\\
&\leq \frac{1}{2}\sum_{k}\int_0^t\int_{\mt} \Big|\nabla_{\bfxi} g_k(\cdot,\bfv^N)\cdot\partial_{\gamma}\bfv^N\Big|^2\dxs\\
&\leq c\,\int_0^t\int_{\mt}|\partial_{\gamma}\bfv^N|^2\dxs+c\,\int_0^t\int_{\mt}|\bfv^N|^2\dxs.
\end{align*}

\subsection{The case $n=2$.}
\emph{Ad $(II)_{5}$.} The crucial impact of our assumption $n=2$ is that $(II)_{5}=0$ which can be established by elementary calculations.
Gathering estimates, we have shown
\begin{align*}
\E\bigg[\sup_{0<t<T}&\int_{\mt}|\nabla\bfv^N(t)|^2\dx+\int_\Q|\nabla_\bfxi\bfF_p(\cdot,\ep(\bfv^N))|^2\dxt\bigg]\\\leq\,&c\,\E\bigg(1+\int_{\mt}\big(|\bfv_0|^2+|\nabla\bfv_0|^2\big)\dx+\E\int_Q\big(|\bff|^2+|\nabla\bff|^2\big)\dxt\bigg)\\
&+ c\,\E\bigg(\int_Q|\ep(\bfv^N)|^{p(\cdot)}\dxt +\int_Q|\bfv^N|^{2}\dxt +\int_Q|\nabla\bfv^N|^{2}\dxt \bigg)\\
&+ c\,\E\int_Q\Big(|\ep(\bfv^N)|^{p(\cdot)}\Big)^{\frac{q}{2}}\dxt, 
\end{align*}
where $q:=\frac{4}{4-\mu}$. \\
The terms in the first line of the right hand side are bounded by assumption. The terms in the second line are bounded by the a priori estimates from Theorem \ref{thm:2.1'} except of the last one. It can, however, be handled by Gronwall's lemma leading to
\begin{align*}
\E\bigg[\sup_{0<t<T}&\int_{\mt}|\nabla\bfv^N|^2\dx+\bigg[\int_\Q|\nabla_{\bfxi}\bfF_p(\cdot,\ep(\bfv^N))|^2\dxt\bigg]\\
&\leq\,c\,\E\bigg(1+\int_Q|\bfF_p(\cdot,\ep(\bfv^N))|^{q}\dxt\bigg).
\end{align*}
By Lipschitz continuity of $p$ we obtain
\begin{align*}
|\nabla\bfF_p(\cdot,\ep(\bfv^N))|&\leq\, |\nabla_{\bfxi}\bfF_p(\cdot,\ep(\bfv^N))|+|\partial_\gamma\bfF_p(\cdot,\ep(\bfv^N))|\\
&\leq\, |\nabla_{\bfxi}\bfF_p(\cdot,\ep(\bfv^N))|+c\,\ln(1+| \ep(\bfv^N)|)(1+|\ep(\bfv^N)|)^{\frac{p(\cdot)}{2}}\\
&\leq\, |\nabla_{\bfxi}\bfF_p(\cdot,\ep(\bfv^N))|+c\,\big(|\bfF_p(\cdot,\ep(\bfv^N))|^\frac{q}{2}+1\big)
\end{align*}
such that
\begin{align}\label{eq:reg}
\begin{aligned}
\E\bigg[\sup_{0<t<T}&\int_{\mt}|\nabla\bfv^N|^2\dx+\bigg[\int_\Q|\nabla\bfF_p(\cdot,\ep(\bfv^N))|^2\dxt\bigg]\\
&\leq\,c\,\E\bigg(1+\int_Q|\bfF_p(\cdot,\ep(\bfv^N))|^{q}\dxt\bigg).
\end{aligned}
\end{align}
Note that $q$ can be chosen arbitrarily close to 2.
The objective of the following is to find a suitable bound for the remaining integral on the right hand side.

By Korn's inequality, $\int_{\mt}|\nabla\bfv^N|^2$ and $\int_{\mt}|\ep(\bfv^N)|^2$ are equivalent. Using the elementary inequality $|\bfF_{p}(\cdot,\bfxi)|^\tau\leq\,c(|\bfxi|^2+1)$ for $\tau=4/p^+$ and Sobolev's embedding
$W^{1,2}(\mt)\hookrightarrow L^{2\sigma}(\mt))$ (with $\sigma=\frac{n}{n-2}$
if $n\geq3$ and $\sigma$ arbitrary for $n=2$) we deduce from \eqref{eq:reg} that
\begin{align}\label{eq:reg'}
\begin{aligned}
\E\bigg[\sup_{0<t<T}&\int_{\mt}|\bfF_p(\cdot,\ep(\bfv^N))|^\tau\dx+\bigg[\int_0^T\bigg(\int_{\mt}|\bfF_p(\cdot,\ep(\bfv^N))|^{2\sigma}\dx\bigg)^{\frac{1}{\sigma}}\dt\bigg]\\
&\leq\,c\,\E\bigg(1+\int_Q|\bfF_p(\cdot,\ep(\bfv^N))|^{q}\dxt\bigg).
\end{aligned}
\end{align}
In order to proceed, we use the interpolation (recall that $\tau>1$ as $p^+<4$)
\begin{align*}
\Big(L^\infty(0,T&;L^\tau(\mt));L^2(0,T;L^{2\sigma}(\mt))\Big)_{\Theta}=L^r(0,T;L^r(\mt)),\\ r&=2+\tau-\frac{\tau}{\sigma},\quad \Theta=1-\frac{2}{r},
\end{align*}
and obtain for $\chi=\frac{2\tau}{2\Theta+\tau(1-\Theta)}$
\begin{align}\label{intera}
\|v\|_r^\chi\leq \|v\|_{L^\infty_tL^\tau_x}^{\chi\Theta}\|v\|_{L^2_tL^{2\sigma}_x}^{\chi(1-\Theta)}\leq\|v\|_{L^\infty_tL^\tau_x}^{\tau}+\|v\|_{L^2_tL^{2\sigma}_x}^{2}.
\end{align}
Combining \eqref{eq:reg'} and \eqref{intera} yields
\begin{align*}
\E\|\bfF_p(\cdot,\ep(\bfv^N))\|_{L^r_{t,x}}^\chi\leq\,c\,\Big(1+\E\|\bfF_p(\cdot,\ep(\bfv^N))\|^q_{L^q_{t,x}}\Big).
\end{align*}
We continue with the interpolation
\begin{align*}
\Big(L^r(Q);L^2(Q)\Big)_{\beta}=L^q(Q),\quad \beta=\frac{r}{q}\frac{q-2}{r-2},
\end{align*}
and obtain
\begin{align*}
\E\|\bfF_p(\cdot,\ep(\bfv^N))\|^q_{L^q_{t,x}}&\leq \E\Big(\|\bfF_p(\cdot,\ep(\bfv^N))\|_{L^r_{t,x}}^{\beta q}\|\bfF_p(\cdot,\ep(\bfv^N))\|_{L^2_{t,x}}^{(1-\beta) q}\Big)\\
&\leq \Big(\E\|\bfF_p(\cdot,\ep(\bfv^N))\|_{L^r_{t,x}}^{\beta q\gamma}\Big)^{\frac{1}{\gamma}}\Big(\E\|\bfF_p(\cdot,\ep(\bfv^N))\|_{L^2_{t,x}}^{(1-\beta) q\gamma'}\Big)^{\frac{1}{\gamma'}}
\end{align*}
using also H\"older's inequality for $\gamma\in(1,\infty)$ arbitrary. By Theorem
\ref{thm:2.1'}, the definition of $\bfF_p$ and the assumptions on the initial law we  find that the second term is uniformly bounded for any choice of $\gamma$. So, we obtain
\begin{align}\label{eq:2202}
\begin{aligned}
\E\|\bfF_p(\cdot,\ep(\bfv^N))\|^\chi_{L^r_{t,x}}&\leq\,c\,(1+\E\|\bfF_p(\cdot,\ep(\bfv^N))\|^q_{L^q_{t,x}})\\&
\leq \,c\,\Big(1+\E\|\bfF_p(\cdot,\ep(\bfv^N))\|_{L^r_{t,x}}^{\beta q\gamma}\Big)^{\frac{1}{\gamma}}.
\end{aligned}
\end{align}
If $\beta q<\chi$ (note that $\beta$ can be made arbitrarily small if we choose $q$ close enough to 2 and $\gamma$ can be chosen arbitrarily close to 1), we finally obtain
\begin{align*}
\E\|\bfF_p(\cdot,\ep(\bfv^N))\|^\chi_{L^r_{t,x}}\leq\,c
\end{align*}
uniformly in $N$. By \eqref{eq:2202} this implies
\begin{align}\label{eq:2202b}
\E\|\bfF_p(\cdot,\ep(\bfv^N))\|^q_{L^q_{t,x}}\leq\,c
\end{align}
uniformly. Inserting this into \eqref{eq:reg} yields the claim.

\subsection{The case $n=3$.}
If $n=3$, the convective term does not vanish. We have to estimate it which is only possible under a restrictive assumption on $p^-$. We have 
\begin{align*}
(II)_5\leq\,\int_0^t\int_{\mt}|\nabla\bfv^N|^3\dxs
\end{align*}
such that we end up with 
\begin{align}\label{eq:reg3}
\begin{aligned}
\E\bigg[\sup_{0<t<T}&\int_{\mt}|\nabla\bfv^N|^2\dx+\bigg[\int_\Q|\nabla\bfF_p(\cdot,\ep(\bfv^N))|^2\dxt\bigg]\\
&\leq\,c\,\E\bigg(1+\int_Q|\ep(\bfv^N)|^{\q}\dxt\bigg),
\end{aligned}
\end{align}
where $\q=\max\{p^++\varrho,3\}$ ($\varrho>0$ is arbitrary) as a counterpart to \eqref{eq:reg}. Using again Sobolev's embedding shows
\begin{align*}
\E\bigg[\sup_{0<t<T}&\int_{\mt}|\ep(\bfv^N)|^2\dx+\bigg[\int_0^T\bigg(\int_{\mt}|\bfF_p(\cdot,\ep(\bfv^N))|^{6}\dx\bigg)^{\frac{1}{3}}\dt\bigg]\\
&\leq\,c\,\E\bigg(1+\int_Q|\ep(\bfv^N)|^{\q}\dxt\bigg).
\end{align*}
We obtain finally
\begin{align}\label{eq:reg'3}
\begin{aligned}
\E\bigg[\sup_{0<t<T}&\int_{\mt}|\ep(\bfv^N)|^2\dx+\bigg[\int_0^T\bigg(\int_{\mt}|\ep(\bfv^N)|^{3p^-}\dx\bigg)^{\frac{1}{3}}\dt\bigg]\\
&\leq\,c\,\E\bigg(1+\int_Q|\ep(\bfv^N)|^{\q}\dxt\bigg).
\end{aligned}
\end{align}
Now we use an interpolation which is quite similar to the two-dimensional case. However, the quantity of interest is now $\ep(\bfv^N)$ instead of $\mathbf{F}_p(\cdot,\ep(\bfv^N))$. Using the interpolation
\begin{align*}
\Big(L^\infty(0,T&;L^2(\mt));L^{p^-}(0,T;L^{3p^-}(\mt))\Big)_{\Theta}=L^r(0,T;L^r(\mt)),\\ r&=\frac{4}{3}+p^-,\quad \Theta=1-\frac{p^-}{r},
\end{align*}
we obtain for $\chi=\frac{3}{5}r$
\begin{align}\label{inter}
\|v\|_r^\chi\leq \|v\|_{L^\infty_tL^2_x}^{\chi\Theta}\|v\|_{L^{p^-}_tL^{3p^-}_x}^{\chi(1-\Theta)}\leq\|v\|_{L^\infty_tL^2_x}^{2}+\|v\|_{L^{p^-}_tL^{3p^-}_x}^{p^-}
\end{align}
such that
\begin{align*}
\E\|\ep(\bfv^N)\|_{L^r_{t,x}}^\chi\leq\,c\,\Big(1+\E\|\ep(\bfv^N)\|^{\q}_{L^{\q}_{t,x}}\Big).
\end{align*}
On account of the interpolation
\begin{align*}
\Big(L^r(Q);L^{p^-}(Q)\Big)_{\beta}=L^{\q}(Q),\quad \beta=\frac{r}{\q}\frac{\q-p^-}{r-p^-},
\end{align*}
we gain similarly to the two-dimensional case
\begin{align*}
\E\|\ep(\bfv^N)\|^{\q}_{L^{\q}_{t,x}}&\leq \E\Big(\|\ep(\bfv^N)\|_{L^r_{t,x}}^{\beta \q}\|\ep(\bfv^N)\|_{L^{p^-}_{t,x}}^{(1-\beta) \q}\Big)\\
&\leq \Big(\E\|\ep(\bfv^N)\|_{L^r_{t,x}}^{\beta \q\gamma}\Big)^{\frac{1}{\gamma}}\Big(\E\|\ep(\bfv^N)\|_{L^{p^-}_{t,x}}^{(1-\beta) \q\gamma'}\Big)^{\frac{1}{\gamma'}}.
\end{align*}
By Theorem \ref{thm:2.1'}
the second term is uniformly bounded and hence
\begin{align}\label{eq:22023}
\E\|\ep(\bfv^N)\|^\chi_{L^r_{t,x}}\leq\,c\,\big(1+\E\|\ep(\bfv^N)\|^{\q}_{L^{\q}_{t,x}}\big)&
\leq \,c\,\Big(1+\E\|\ep(\bfv^N)\|_{L^r_{t,x}}^{\beta \q\gamma}\Big)^{\frac{1}{\gamma}}.
\end{align}
Now we have to check that $\beta \q< \chi$. This is equivalent to $\q<p^-+\frac{4}{5}$ which follows from our assumption $\frac{11}{5}<p^-\leq p^+\leq p^-+\frac{4}{5}$.
So, the proof can be finished as before if we chose $\gamma$ close enough to 1.
\end{proof}

\subsection{Compactness}
As in \eqref{eq:Hclaim4} we have again
\begin{align}\label{eq:Hclaim4b}
\sup_{N\in\mathbb{N}}\E\left[\lVert{\bfv^{N}\rVert}_{C^\mu([0,T];\sobo_{\Div}^{-\ell,p_{0}}(\tn)^{n})}\right]<\infty
\end{align}
for certain $\mu>0$, $\ell\in\mathbb N$ and $p_0>1$.
In view of compactness, let us now define the path space 
\begin{align}\label{eq:pathspace}
\begin{split}
\mathcal{X}&:=\mathcal X_{\bfv}\otimes \mathcal X_{\bfF}\otimes \mathcal X_{p}\otimes \mathcal X_{\bff}\otimes \mathcal X_W,
\end{split}
\end{align}
where\footnote{$(X,w)$ denotes a Banach space equipped with the weak topology.}
\begin{align*}
\mathcal X_{\bfv}&:= C([0,T];\sobo_{\Div}^{-\ell,p_{0}}(\tn))\cap L^{2}(0,T;W^{1,2}_{\Div}(\mt)),\\
\mathcal X_{\bfF}&:= \big(L^2(0,T;W^{1,2}(\mt)),w\big),\\
\mathcal X_{p}&:= \hold^0([0,T]\times \mt),\\
\mathcal X_{\bff}&:= L^2(0,T;W^{1,2}(\mt)),\\
\mathcal X_W&:=\hold([0,T];\mathfrak{U}_{0}).
\end{align*}
We obtain the following.
\begin{proposition}\label{prop:bfutightness''}
The set $\{\mathcal{L}[\bfv^N,\bfF_p(\cdot,\ep(\bfv^N)),p,\bff,W];\,N\in\mathbb N\}$ is tight on $\mathcal{X}$.
\end{proposition}
\begin{proof}
We recall an interpolation result of Aubin--Lions--type due to Amann \cite{Am} to conclude that  
\begin{align}\label{eq:cpctAmann}
\begin{aligned}
\lebe^{\infty}(0,T;\sobo_{\Div}^{1,2}(\tn))&\cap C^\mu([0,T];W_{\Div}^{-\ell,p_0}(\mt))\cap \lebe^{p^-}(0,T;\sobo_{\Div}^{1,p^-}(\tn))\\ &\hookrightarrow\hookrightarrow \lebe^{2}(0,T;\sobo_{\Div}^{1,2}(\tn)).
\end{aligned}
\end{align}
On the other hand, Ascoli-Arzel\'{a}'s theorem yields
\begin{align*}
C^\mu([0,T];W^{-\ell,p_0}_{\Div}(\mt))\hookrightarrow \hookrightarrow C([0,T];W_{\Div}^{-\ell,p_0}(\mt)).
\end{align*}
So, we obtain tightness of $\bfv^N$ on $\mathcal X_\bfv$ from \eqref{eq:Hclaim4b},
Theorems \ref{thm:2.1'} and \ref{thm:2.1''} and Tschebyscheff's inequality.
Tightness of $\bfF_p(\cdot,\ep(\bfv^N))$ on $\mathcal X_\bfF$ follows immediately from
Theorem \ref{thm:2.1'} and \ref{thm:2.1''}.
Finally the laws of $p$, $\bff$ and $W$ are tight as before in Section \ref{subsec:comp}.
\end{proof}

Accordingly, we apply Jakubowski's extension of Skorokhod's
theorem (see \cite{jakubow}).
We infer the following result.

\begin{proposition}\label{prop:skorokhodb}
There exists a complete probability space $(\tilde\Omega,\tilde\mf,\tilde\prst)$ with $\mathcal{X}$-valued Borel measurable random variables $(\tilde\bfv^N,\tilde \bfF^N,\tilde p^N,\tilde \bff^N,\tilde W^N)$, $N\in\N$, and $(\tilde\bfv,\tilde \bfF,\tilde p,\tilde \bff,\tilde W)$ such that (up to a subsequence)
\begin{enumerate}
 \item the law of $(\tilde\bfv^N,\tilde \bfF^N,\tilde p^N,\tilde \bff^N,\tilde W^N)$ on $\mathcal{X}$ is given by $\mathcal{L}[\bfv^N,\bfF_p(\cdot,\ep(\bfv^N)), p,\bff,W]$, $N\in\N$,
\item the law of $(\tilde\bfv,\tilde \bfF,\tilde p,\bff,\tilde W)$ on $\mathcal{X}$ is a Radon measure,
 \item $(\tilde\bfv^N,\tilde \bfF^N,\tilde p^N,\tilde \bff^N,\tilde W^N)$ converges $\,\tilde{\prst}$-almost surely to $(\tilde\bfv,\tilde \bfF,\tilde p,\tilde \bff,\tilde W)$ in the topology of $\mathcal{X}$, i.e.
\begin{align} \label{wWS116b}
\begin{aligned}
\tilde \bfv^N &\to \tilde\bfv \ \mbox{in}\ \hold([0,T];W_{\Div}^{-\ell,p_0}(\mt)) \ \tilde\p\mbox{-a.s.}, \\
\tilde \bfv^N &\to \tilde\bfv \ \mbox{in}\ L^2(0,T;W^{1,2}_{\Div}(\mt)) \ \tilde\p\mbox{-a.s.}, \\
\tilde \bfF^N &\rightharpoonup \tilde \bfF \ \mbox{in}\ L^2(0,T;W^{1,2}(\mt)) \ \tilde\p\mbox{-a.s.}, \\
\tilde p^N &\to \tilde p \ \mbox{in}\ \hold^0([0,T]\times\mt) \ \tilde\p\mbox{-a.s.}, \\
\tilde \bff^N &\to \tilde \bff \ \mbox{in}\ L^2(0,T;W^{1,2}(\mt)) \ \tilde\p\mbox{-a.s.}, \\
\tilde W^N &\to \tilde W \ \mbox{in}\ \hold([0,T]; \mathfrak{U}_0 )\ \tilde\p\mbox{-a.s.}
\end{aligned}
\end{align}
\end{enumerate}
\end{proposition}

The equality of laws from Proposition \ref{prop:skorokhod} implies immediately that \db{$\tilde\bfF^N=\bfF_{\tilde p^N}(\cdot,\ep(\bfv^N))$}. \db{Using the convergences from \eqref{wWS116b} we obtain
\begin{align}\label{eq:0701}
\tilde\bfF=\bfF_{\tilde p}(\cdot,\ep(\tilde\bfv)).
\end{align}
}
Also, the uniform estimates from Theorems \ref{thm:2.1'} and \ref{thm:2.1''} continue to holds on the new probability space. The proof of Theorem \ref{thm:main2} can now be completed as in Section \ref{sec:weak}.

\subsection{\db{Strong stochastically strong solutions}}
The existence of a strong pathwise solution follows now along the lines of the proof of Theorem \ref{thm:main2} with some minor modifications.
The most important change is that the classical Gy\"ongy-Krylov argument does not apply as the path space $\mathcal X$ is not Polish anymore due to the weak topology on $\mathcal X_{\bfF}$. A generalization which applies to the very general class of sub-Polish spaces (including Banach spaces with weak topologies) can be found in \cite[Chapter 2, Theorem 2.10.3]{BFHbook}.
We consider the collection of joint laws of $(\bfX^n,\bfX^m,p,\bff,W)$, where
\begin{align*}
\bfX^n&=(\bfv^{N_n},\bfF_p(\cdot,\ep(\bfv^{N_n}))),\quad
\bfX^m=(\bfv^{N_m},\bfF_p(\cdot,\ep(\bfv^{N_m}))),
\end{align*}
on the extended path space
$$\mathcal{X}^J=(\mathcal X_{\bfv}\times\mathcal X_\bfF)^2\otimes \mathcal X_{p}\otimes \mathcal X_{\bff}\times\mathcal{X}_W.$$ 
As in Proposition \ref{prop:tight}
we obtain tightness of the set
$$\{\mathcal L[\bfX^n,\bfX^m,p,\bff,W];\,n,m\in\mn\}$$
on $\mathcal{X}^J$.
Let $(\bfX^{n_k},\bfX^{m_k},p,\bff,W)$be an arbitrary subsequence. By the Jakubowski--Skorokhod theorem \cite{jakubow} we infer (for a further subsequence but without loss of generality we keep the same notation) the existence of a probability space $(\bar{\Omega},\bar{\mf},\bar{\prst})$ with a sequence of random variables $(\hat\bfX^{n_k},\check\bfX^{m_k},\bar{p}_k,\bar{\bff}_k,\bar{W}_k)$
 with
\begin{align*}
\hat\bfX^{n_k}&=(\hat\bfv^{{n_k,}},\hat\bfF^{{n_k}}),\quad k\in\mn,\\
\hat\bfX^{m_k}&=(\check\bfv^{{m_k,}},\check\bfF^{{m_k}}),\quad k\in\mn,
\end{align*}
converging almost surely in $\mathcal{X}^J$ to a random variable $(\hat\bfX,\check\bfX,\bar p,\bar\bff,\bar{W})$.
with
\begin{align*}
\hat\bfX&=(\hat\bfv,\hat\bfF),\quad
\check\bfX=(\hat\bfv,\hat\bfF).
\end{align*}
As before \db{in \eqref{eq:0701}} it follows that
\begin{align}\label{eq:911}
\hat\bfF=\bfF_{\bar p}(\cdot,\ep(\hat\bfv)),\quad \check\bfF=\bfF_{\bar p}(\cdot,\ep(\check\bfv)).
\end{align}
%
As in \eqref{eq:LIMIT} we can
show that $(\hat\bfv,\bar p,\bar \bff,\bar W)$ and $(\check\bfv,\bar p,\bar \bff,\bar W)$ are weak martingale solutions to \eqref{0.4}--\eqref{0.3} defined on the same stochastic basis $(\bar{\Omega},\bar{\mf},(\bar{\mf}_t),\bar{\prst})$.
We apply the pathwise uniqueness result from Proposition \ref{prop:unique} to conclude
\begin{align*}
\mathcal L&[\hat\bfX,\check\bfX,\bar{W}]\Big((\bfX_1,\bfX_2,p,\bff,W)\in\mathcal X^J:\;\bfX_1=\bfX_2\Big)\\
&\quad=\bar{\prst}\Big((\hat\bfv,\hat\bfF)=(\check\bfv,\check\bfF)\Big)
=\bar{\prst}(\hat{\bfv}=\check{\bfv})=1.
\end{align*}
Now, \cite[Chapter 2, Theorem 2.10.3]{BFHbook}
implies that the original sequence $\bfv^N$ defined on the initial probability space converges in probability in the topology of $\mathcal{X}_{\bfv}$ to the random variable $\bfv$. Therefore, we finally deduce that $\bfv$ is a \db{strong stochastically strong solution} to \eqref{0.4}--\eqref{0.3}. \db{Note that the pressure terms can be recovered as in \eqref{eq:pressure} (see the explanations below \eqref{eq:pressure} for the regularity of the pressure terms).}
 The proof of Corollary \ref{cor:strong} is hereby complete.
\hfill $\Box$


\begin{thebibliography}{M}
\bibitem{Am} H. Amann (2000): \emph{Compact embeddings of vector-valued Sobolev and Besov
spaces.} Glass. Mat., III. Ser. 35(55), 161--177.
\bibitem{AcMiSe} E. Acerbi, G. Mingione, G. Seregin (2004): \emph{Regularity results for parabolic systems related to a class of non-Newtonian fluids.} Ann. I. H. Poincar\'e -- AN 21, 25--60.

\bibitem{BaVaWiZi} C. Bauzet, G. Vallet, P. Wittbold, A. Zimmermann  (2013): \emph{On a $p(t,x)$-Laplace evolution equation with a stochastic force.} SPDE: Anal. Comp. 1(3), 552--570.
\bibitem{Br2} D. Breit (2015): \emph{Existence theory for stochastic power law fluids.} J. Math. Fluid Mech.
17, 295--326.
\bibitem{BFHbook} D. Breit, E. Feireisl, M. Hofmanov\'{a}  (2018): \emph{Stochastically forced compressible fluid flows.} De Gruyter Series in Applied and Numerical Mathematics. De Gruyter, Berlin/Munich/Boston.
\bibitem{BrHo} D. Breit, M. Hofmanov\'a (2016): \emph{ Stochastic Navier--Stokes equations for compressible fluids.} Indiana Univ. Math. J. 65, 1183--1250. 
\bibitem{CHL1} C. Chen, J. Hong, L. Zhang (2016): \emph{Preservation of physical properties of stochastic Maxwell equations with additive noise via stochastic multi-symplectic methods.}
J. Comp. Phys. 306, 500--519.
\bibitem{CHL2} J. Hong, L. Ji, L. Zhang, J. Cai (2017):\emph{An energy-conserving method for stochastic Maxwell equations with multiplicative noise.}
J. Comp. Phys. 351, 216--229.
 \bibitem{PrZa}  G. Da Prato, J. Zabczyk (1992): \emph{Stochastic equations in infinite dimensions.} Encyclopedia Math. Appl., vol. 44, Cambridge University Press, Cambridge. 
\bibitem{debussche1} A. Debussche, N. Glatt-Holtz, R. Temam (2011): \emph{Local Martingale and Pathwise Solutions for an Abstract Fluids Model.} Physica D 14-15, 1123--1144.
\bibitem{Di}
L.~Diening (2002): \emph{Theoretical and numerical results for electro-rheological
  fluids}. Ph.D. thesis, Albert-Ludwigs-Universit\"at, Freiburg.
\bibitem{DiHaHaRu} 
  L. Diening, P. H\"ast\"o, P. Harjulehto, M. R{\r u}{\v z}i{\v c}ka (2011):
  \emph{Lebesgue and Sobolev spaces with variable exponents}. Springer
  Lecture Notes, vol. 2017, Springer-Verlag, Berlin.
	\bibitem{DMS} L. Diening, J. M\'alek, M. Steinhauer (2008): \emph{On Lipschitz truncations of Sobolev functions (with
variable exponent) and their selected applications.} ESAIM: Control, Optimisation and Calculus of Variations
  14(2), 211--232
\bibitem{DNR} L. Diening, P. N\"agele, M. R{\r u}{\v z}i{\v c}ka (2012): \emph{Monotone operator theory for unsteady problems in variable
exponent spaces.} Complex Variables Elliptic Equ. 57, 1209--1231
	\bibitem{DRW} L. Diening, M. R\r{u}\v{z}i\v{c}ka, J. Wolf (2010): \emph{Existence of weak solutions for unsteady motions of generalized Newtonian fluids.} Ann. Sc. Norm. Sup. Pisa Cl. Sci. IX(5), 1-46.
\bibitem{Fl} F. Flandoli (2008):
\newblock \emph{An introduction to 3D stochastic fluid dynamics.}
\newblock In SPDE in hydrodynamic:
recent progress and prospects. Lecture Notes in Math. 1942, 51--150. Springer,
Berlin.
\bibitem{FlGa} F. Flandoli, D. G\c{a}tarek (1995): \emph{Martingale and stationary solutions for stochastic Navier--Stokes equations.} Probab. Theory Rel.
Fields 102, 367--391.
	\bibitem{GT} D. Gilbarg, N. S. Trudinger (2011): \emph{Elliptic partial differential equations of second order.} Reprint of the 1998 edition. Classics in Mathematics. Springer-Verlag, Berlin.

\bibitem{krylov} I. Gy\"{o}ngy, N. Krylov (1996): \emph{Existence of strong solutions for It\^o's stochastic equations via approximations.} Probab. Theory Rel. Fields 105(2), 143-158.

\bibitem{IkWa}	N. Ikeda,  S. Watanabe (1989): \emph{Stochastic differential equations and diffusion processes.}
2nd ed. North-Holland Mathematical Library 24. North-Holland, Amsterdam.
\bibitem{KS} I. Karatzas, S. E. Shreve (1998): \emph{Brownian motion and stochastic calculus.} Springer.

\bibitem{KukShi}
S.~Kuksin and A.~Shirikyan (2012):
\newblock {\em Mathematics of two-dimensional turbulence}, volume 194 of {\em
  Cambridge Tracts in Mathematics}.
\newblock Cambridge University Press, Cambridge.

\bibitem{jakubow} A. Jakubowski (1997/1998): \emph{The almost sure Skorokhod representation for subsequences in nonmetric spaces.} Teor. Veroyatnost. i Primenen 42(1), 209--216/translation in Theory Probab. Appl. 42(1), 167--174.
   \bibitem{MNR} J. M{\'a}lek, J. Ne\v{c}as, M. R\r{u}\v{z}i\v{c}ka (1993): \emph{On the non-Newtonian incompressible fluids.} Math. Mod. Meth. Appl. Sci. 3(1), 35--63.
   \bibitem{MNRR} J. M{\'a}lek, J. Nec\v{a}s, M. Rokyta, M. R\r{u}\v{z}i\v{c}ka (1996):   \emph{Weak and measure valued solutions to evolutionary PDEs.} Chapman \& Hall, London-Weinheim-New York.

\bibitem{PrRo}  C. Pr\'{e}v\^{o}t, M. R\"ockner (2007): \emph{A concise course on stochastic partial differential equations.} Lecture Notes in Mathematics 1905. Springer, Berlin.
\bibitem{RaRu1}
K.R. Rajagopal, M.~R{\r u}{\v z}i{\v c}ka (1996):
\emph{On the modeling of electro-rheological materials.}
 Mech. Res. Commun. 23(4), 401--407.

\bibitem{RaRu2}
K.R. Rajagopal, M.~R{\r u}{\v z}i{\v c}ka (2001):
 \emph{Mathematical modeling of electro-rheological materials.}
Cont. Mech. and Thermodyn. 13, 59--78.


\bibitem{Ro} 
M.~Romito. (2016):
\newblock \emph{Some probabilistic topics in the Navier--Stokes equations.} 
\newblock Recent progress in the theory of the Euler and Navier--Stokes equations.
\newblock London Math. Soc. Lecture Note Ser. 430, 175--232,  Cambridge Univ. Press, Cambridge.

\bibitem{Ru}
  M.~R{\r u}{\v{z}}i{\v{c}}ka (2000):  {\em Electro-rheological
    fluids: modeling and mathematical theory}, volume 1748 of {\em
    Lecture Notes in Mathematics}.   Springer-Verlag, Berlin.

\bibitem{schm} B. Schmalfuss (1997): \emph{Qualitative properties for the stochastic Navier-Stokes equation} Nonlinear Anal. 28(9), 1545--1563.

\bibitem{2015arXiv150400951S}
S.~{Smith} (2017):
\newblock {Random perturbations of viscous compressible fluids: global
  existence of weak solutions}.
\newblock SIAM J. Math. Anal., 49(6), 4521--4578.

\bibitem{Sm} J. S. Smagorinsky (1963): \emph{General circulation experiments with the primitive equations. I. The basic
experiment.} Mon. Weather Rev. 91(3), 99--164.

\bibitem{TeYo} Y. Terasawa, N. Yoshida (2011): \emph{Stochastic power-law fluids:
existence and uniqueness of weak solutions.} Ann. Appl. Prob. 21(5), 1827--1859.

\bibitem{VaWiZi} G. Vallet, P. Wittbold, A. Zimmermann (2016): \emph{On a stochastic evolution equation with random growth conditions.} Stoch. Partial Differ. Equ. Anal. Comput. 4(2), 246--273.

\bibitem{Wi} A.M. Winslow (1949): \emph{Induced fibration of suspensions.} J. Appl. Physics 20, 1137--1140.

\bibitem{Yo} N. Yoshida (2012): \emph{Stochastic shear thickenning fluids: strong convergence of the
Galerkin approximation and the energy inequality. } Ann. Appl. Prob. 22(3), 1215--1242.
\end{thebibliography}
\end{document}